\documentclass[reqno,12pt]{amsart}
\usepackage{amsmath,amssymb,amsthm,graphicx,mathrsfs,url,bbm}
\usepackage{wrapfig}
\usepackage{enumitem}
\usepackage{mathtools}
\usepackage[utf8]{inputenc}
 \usepackage[T1]{fontenc}
\usepackage[usenames,dvipsnames]{color}
\usepackage[colorlinks=true,linkcolor=Red,citecolor=Green]{hyperref}
\usepackage[super]{nth}
\usepackage[open, openlevel=2, depth=3, atend]{bookmark}
\hypersetup{pdfstartview=XYZ}
\usepackage[font=footnotesize]{caption}
\usepackage{a4wide}
\usepackage[stretch=10]{microtype}

\usepackage{epstopdf}

\usepackage{hyperref}

\theoremstyle{plain}
\newtheorem{theorem}{Theorem}[section]

\newtheorem{lemma}[theorem]{Lemma}
\newtheorem{proposition}[theorem]{Proposition}

\theoremstyle{definition}

\theoremstyle{remark}
\newtheorem{remark}[theorem]{Remark}

\numberwithin{equation}{section}

\newcommand{\R}{\mathbb{R}}

\newcommand{\Z}{\mathbb{Z}}

\newcommand{\M}{\mathcal{M}}
\newcommand{\N}{\mathbb{N}}

\newcommand{\V}{\mathbb{V}}

\newcommand{\Ss}{\mathbb{S}}
\newcommand{\eps}{\varepsilon}
\newcommand{\ke}{\text{ker }}

\newcommand{\mc}{\mathcal}

\DeclareMathOperator{\vol}{vol}

\DeclareMathOperator{\Tr}{Tr}
\DeclareMathOperator{\Op}{Op}

\DeclareMathOperator{\ran}{ran}

\DeclareMathOperator{\Div}{div}

\DeclareMathOperator{\sol}{sol}

\DeclarePairedDelimiter{\abs}{\lvert}{\rvert}
\DeclarePairedDelimiter{\norm}{\lVert}{\rVert}

\newcommand{\be}{\begin{equation}}
\newcommand{\ee}{\end{equation}}

% ...

% ...

\title
[X-ray transform on Anosov manifolds]
{Classical and microlocal analysis of the X-ray transform on Anosov manifolds}

\author{Sébastien Gouëzel}
\address{Laboratoire Jean Leray, CNRS UMR 6629,
Université de Nantes, 2 rue de la
Houssinière,
44322 Nantes, France}
\email{sebastien.gouezel@univ-nantes.fr}

\author{Thibault Lefeuvre}

\address{Laboratoire de Mathématiques d’Orsay, Univ. Paris-Sud, CNRS, Université Paris-Saclay, 91405 Orsay, France}

\email{thibault.lefeuvre@u-psud.fr}

\begin{document}

\begin{abstract}
We complete the
microlocal study of the geodesic X-ray transform on Riemannian manifolds
with Anosov geodesic flow initiated by Guillarmou in \cite{Guillarmou-17-1}
and pursued by Guillarmou and the second author in \cite{Guillarmou-Lefeuvre-18}.
We prove new stability estimates and clarify some properties of the
operator $\Pi_m$ — the generalized X-ray transform. These estimates rely on a 
refined version of the Livsic theorem for Anosov flows, especially on a new quantitative finite time Livsic
theorem.
\end{abstract}

\maketitle

\section{Introduction}

Let $\mc{M}$ be a smooth closed $(n+1)$-dimensional manifold endowed with a
vector field $X$ generating a complete flow $(\varphi_t)_{t \in \R}$. We
assume that the flow $(\varphi_t)_{t \in \R}$ is transitive and
Anosov in the sense that there exists a continuous flow-invariant splitting
\be \label{eq:split} T_x(\mc{M}) = \R X(x) \oplus E_u(x) \oplus E_s(x), \ee
where $E_s(x)$ (resp.\ $E_u(x)$) is the \textit{stable} (resp.\
\textit{unstable}) vector space at $x \in \mc{M}$, and a smooth Riemannian
metric $g$ such that
\be \label{equation:hyperbolicite} \begin{array}{c} |d\varphi_t(x) \cdot v|_{\varphi_t(x)} \leq C e^{-\lambda t} |v|_{x}, ~~ \forall t > 0, v \in E_s(x) \\
|d\varphi_t(x) \cdot v|_{\varphi_t(x)} \leq C e^{-\lambda |t|} |v|_{x}, ~~
\forall t < 0, v \in E_u(x),\end{array} \ee for some uniform constants
$C,\lambda > 0$. The norm, here, is $|\cdot|_x \coloneqq
g_x(\cdot,\cdot)^{1/2}$. The dimension of $E_s$ (resp.\ $E_u$) is denoted
by $n_s$ (resp.\ $n_u$). As a consequence, $n+1=1+n_s+n_u$ (where the $1$
stands for the \textit{neutral} direction, that is the direction of the
flow). The case we will have in mind will be that of a geodesic flow on the
unit tangent bundle of a smooth Riemannian manifold
$(M,g)$ with negative sectional curvature.

\subsection{X-ray transform on \texorpdfstring{$\M$}{M}}

We denote by $\mc{G}$ the set of closed orbits of the flow and for $f \in
C^0(\M)$, its X-ray transform $If$ is defined by:
\[ \mc{G} \ni \gamma \mapsto If(\gamma) \coloneqq \langle \delta_{\gamma} , f \rangle = \dfrac{1}{\ell(\gamma)} \int_0^{\ell(\gamma)} f(\varphi_t x) dt, \]
where $x \in \gamma$, $\ell(\gamma)$ is the length of $\gamma$.

The Livsic theorem characterizes the kernel of the X-ray transform for a
hyperbolic flow: the latter is reduced to the coboundaries, which are the
functions of the form $f=Xu$, where $u$ is a function defined on $\mc{M}$
whose regularity is prescribed by that of $f$. This result was initially
proved by Livsic \cite{Livsic-72} in Hölder regularity: if $f \in
C^\alpha(\M)$ is such that $If = 0$, then there exists $u \in
C^\alpha(\M)$, differentiable in the flow direction, such that $f = Xu$ and
$u$ is unique up to an additive constant. There is also a version of the
Livsic theorem in smooth regularity due to De la Llave-Marco-Moriyon
\cite{DeLaLlave-Marco-Moryon-86}. Much more recently, Guillarmou
\cite[Corollary 2.8]{Guillarmou-17-1} proved a version of the Livsic
theorem in Sobolev regularity which implies the theorem of
\cite{DeLaLlave-Marco-Moryon-86}.

It is also rather natural to expect other versions of the Livsic theorem to
hold. For instance, if we modify the condition $If = 0$ by $If \geq 0$, is
it true that $f \geq Xu$, for some well-chosen function $u$ (positive
Livsic theorem)? And if $\|If\|_{\ell^\infty} \coloneqq \sup_{\gamma \in
\mc{G}} |If(\gamma)| \leq \eps$, can one write $f = Xu + h$, where some
norm of $h$ is controlled by a power of $\eps$ (approximate Livsic
theorem)? Eventually, what can be said if $If(\gamma)=0$ for all closed
orbits $\gamma$ of length $\leq L$ (finite Livsic theorem)?

The positive Livsic theorem for Anosov flows was proved by Lopes-Thieullen
\cite{Lopes-Thieullen-05} with an explicit control of a Hölder norm of the coboundary $Xu$ in terms of a
norm of $f$.

\begin{theorem}[Lopes-Thieullen]
\label{theorem:livsic-positif} Let $0 < \alpha \leq 1$. There exist $0 <
\beta \leq \alpha$, $C > 0$ such that: for all functions $f \in
C^\alpha(\mc{M})$, there exist $u \in C^\beta(\mc{M})$, differentiable in the
flow-direction with $Xu \in C^\beta(\mc{M})$ and $h \in C^\beta(\mc{M})$,
such that $f = Xu + h + m(f)$, with $h \geq 0$ and $m(f)=\inf_{\gamma \in
\mc{G}} If(\gamma)$. Moreover, $\|Xu\|_{C^\beta} \leq C\|f\|_{C^\alpha}$.
\end{theorem}

In this article, we prove a finite approximate version of the Livsic theorem, as follows.

\begin{theorem}
\label{theorem:livsic-approche} Let $0 < \alpha \leq 1$. There exist $0 <
\beta \leq \alpha$ and $\tau, C>0$ with the following property. Let $\eps
> 0$. Consider a function $f\in C^\alpha(\mc{M})$ with $\|f\|_{C^\alpha(\mc{M})} \leq 1$
such that $\abs{If(\gamma)} \leq \eps$ for all $\gamma$ with $\ell(\gamma)
\leq \eps^{-1/2}$. Then there exist $u \in C^\beta(\mc{M})$ differentiable in
the flow-direction with $Xu \in C^\beta(\mc{M})$ and $h \in C^\beta(\mc{M})$,
such that $f = Xu + h$. Moreover, $\|u\|_{C^\beta} \leq C$ and
$\|h\|_{C^\beta} \leq C \eps^\tau$.
\end{theorem}

We note that a rather similar result had already been obtained by S. Katok \cite{Katok-90} in the particular case of a contact Anosov flow on a $3$-manifold.

The assumptions of Theorem \ref{theorem:livsic-approche} hold in particular if $\|If\|_{\ell^\infty} =
\sup_{\gamma \in \mc{G}} |If(\gamma)| \leq \eps$. Under the assumptions of
the theorem (only mentioning the closed orbits of length at most
$\eps^{-1/2}$), the decomposition $f = Xu+h$ also gives a global control on
$\|If\|_{\ell^\infty}$, of the form
\begin{equation}
\label{finite_livsic}
  \|If\|_{\ell^\infty} \leq C \eps^\tau.
\end{equation}
Indeed, if one integrates $f=Xu+h$ along a closed orbit of any length, the
contribution of $Xu$ vanishes and one is left with a bound $\norm{h}_{C^0}
\leq C \eps^\tau$. The bound~\eqref{finite_livsic} holds in particular if
$If(\gamma)=0$ for all $\gamma$ with $\ell(\gamma)\leq \eps^{-1/2}$. This
statement illustrates quantitatively the fact that the quantities
$If(\gamma)$ for different $\gamma$ are far from being independent.

\begin{remark}
In Theorem~\ref{theorem:livsic-approche},
the constants $\beta, C,\tau$ depend on the Anosov flow under consideration,
but in a locally uniform way: given an Anosov flow, one can find such
parameters that work for any flow in a neighborhood of the initial flow. The
local uniformity can be checked either directly from the proof, or using a
(Hölder-continuous) orbit-conjugacy between the initial flow and the
perturbed one.
\end{remark}

\begin{remark}
It could be interesting to extend the positive and the finite approximate Livsic theorems to other regularities like $H^s$ spaces for $s > \frac{n+1}{2}$ but we were unable to do so.
\end{remark}

\subsection{X-ray transform for the geodesic flow}

If $(M,g)$ is a smooth closed Riemannian manifold, we set $\mc{M} \coloneqq
SM$, the unit tangent bundle, and denote by $X$ the geodesic vector field
on $SM$. We will always assume that the geodesic flow is Anosov on $SM$ and
we say that $(M,g)$ is an \emph{Anosov Riemannian manifold}. It
is a well-known fact that a negatively-curved manifold has Anosov geodesic
flow. We will denote by $\mc{C}$ the set of free homotopy classes on $M$:
they are in one-to-one correspondence with the set of conjugacy classes of
$\pi_1(M,\bullet)$. If $(M,g)$ is Anosov, we know by \cite{Klingenberg-74}
that given a free homotopy class $c \in \mc{C}$, there exists a unique
closed geodesic $\gamma \in \mc{G}$ belonging to the free homotopy class
$c$. In other words, $\mc{G}$ and $\mc{C}$ are in one-to-one
correspondence. As a consequence, we will rather see the X-ray transform as
a map $I^g : C^0(SM) \rightarrow \ell^\infty(\mc{C})$ and we will drop the
index $g$ if the context is clear.

If $f \in C^\infty(M,\otimes^m_S T^*M)$ is a symmetric tensor, then by
\S\ref{appendix:tensors}, we can see $f$ as a function $\pi_m^*f \in
C^\infty(SM)$, where $\pi_m^*f(x,v) \coloneqq f_x(v,...,v)$. The X-ray
transform $I_m$ of $f$ is simply defined by $I_m f \coloneqq I \circ
\pi_m^*f$. In other words, it consists in integrating the tensor $f$ along
closed geodesics by plugging $m$-times the speed vector in $f$. This map $I_m$ may appear in different contexts. In particular,
$I_2$ is well-known to be the differential of the marked length spectrum and
it was studied in \cite{Guillarmou-Lefeuvre-18} to prove its rigidity, thus
partially answering the conjecture of Burns-Katok \cite{Burns-Katok-85}. 

The natural operator of derivation of symmetric tensors is $D \coloneqq
\sigma \circ \nabla$, where $\nabla$ is the Levi-Civita connection and
$\sigma$ is the operator of symmetrization of tensors (see
\S\ref{appendix:tensors}). Any smooth tensor $f \in C^\infty(M,\otimes^m_S
T^*M)$ can be uniquely decomposed as $f = Dp+h$, where $p \in
C^\infty(M,\otimes^{m-1}_S T^*M)$ and $h \in C^\infty(M,\otimes^m_S T^*M)$
is a \emph{solenoidal tensor} i.e., a tensor such that $D^*h = 0$, where
$D^*$ is the formal adjoint of $D$. We say that $Dp$ is the
\emph{potential part} of the tensor $f$. We will see that $I_m(Dp) = 0$. In
other words, the potential tensors are always in the kernel of the X-ray
transform. We will say that $I_m$ is \emph{solenoidal injective} or in
short s-injective if injective when restricted to
\[ C_{\sol}^\infty(M,\otimes^m_S T^*M) \coloneqq C^\infty(M,\otimes^m_S T^*M) \cap \ker(D^*) \]
Note that we will often add an index $\sol$ to a functional space on
tensors to denote the fact that we are considering the intersection with
$\ker D^*$.

It is conjectured that $I_m$ is s-injective for all Anosov Riemannian
manifolds, in any dimension and without any assumption on the curvature.
Under the additional assumption that the sectional curvatures are
non-positive, the \emph{Pestov energy identity} allows to show injectivity
(see \cite{Guillemin-Kazhdan-80} and \cite{Croke-Sharafutdinov-98} for the
original proofs). Without any assumption on the curvature, this is still
true for surfaces by \cite{Paternain-Salo-Uhlmann-14-2} and
\cite{Guillarmou-17-1}. In higher dimensions, it holds for $m=0,1$ (see
\cite{Dairbekov-Sharafutdinov-03}) but remains an open question for higher
order tensors without any assumption on the curvature. However, it is
already known that $C_{\sol}^\infty(M,\otimes^m_S T^*M) \cap \ker(I_m)$ is
finite-dimensional.

We will also prove a stability estimate on $I_m$. 

\begin{theorem}
\label{theorem:xray} Assume $I_m$ is s-injective. Then for all $0 < \beta < \alpha < 1$, there exists $\theta_1 \coloneqq\theta(\alpha,\beta) > 0$ and $C\coloneqq C(\alpha,\beta) > 0$ such that: if $f \in C_{\sol}^\alpha(M,\otimes^m_S T^*M)$ is a solenoidal symmetric $m$-tensor such that $\|f\|_{C^\alpha} \leq 1$, then $\|f\|_{C^\beta} \leq C \|I_mf\|_{\ell^\infty}^{\theta_1}$.
\end{theorem}


Actually, if $I_m$ is not known to be injective, one still has the previous
estimate by taking $f$ solenoidal and orthogonal to the kernel of $I_m$.
Combining this estimate with Theorem \ref{theorem:livsic-approche} (and more
specifically~\eqref{finite_livsic}), we immediately obtain the following 

\begin{theorem}
\label{theorem:injectivite-xray-fini}
Assume $I_m$ is s-injective. Then for all $0 < \beta < \alpha < 1$, there exists $\theta_2 \coloneqq\theta(\alpha,\beta) > 0$ and $C\coloneqq C(\alpha,\beta) > 0$ such that for any $L > 0$ large enough: if $f \in C_{\sol}^\alpha(M,\otimes^m_S T^*M)$ is a solenoidal symmetric $m$-tensor such that $\|f\|_{C^\alpha} \leq 1$, and $I_mf(\gamma)=0$ for all closed geodesics $\gamma \in \mc{C}$ such that $\ell(\gamma) \leq L$, then $\|f\|_{C^\beta} \leq CL^{-\theta_2}$.
\end{theorem}


Even in the case where $f \in C^\alpha(M)$ is a function on $M$, this result seemed to be previously unknown. \\

\noindent \textbf{Acknowledgements:} We warmly thank Yannick Guedes Bonthonneau and Colin Guillarmou for fruitful discussions. T.L. has received funding from the European Research Council (ERC) under the European Union’s Horizon 2020 research and innovation programme (grant agreement No. 725967).

\section{On symmetric tensors}

\label{appendix:tensors}

We describe elementary properties of symmetric tensors on Riemannian
manifolds. This is a background section for which we also refer to \cite{Guillemin-Kazhdan-80-2, Dairbekov-Sharafutdinov-10}.

\subsection{Definitions and first properties}

\subsubsection{Symmetric tensors in Euclidean space}

\label{ssection:euclidien}

Let $E$ be a Euclidean $(n+1)$-dimensional vector space endowed with a
metric $g$ and let $(\mathbf{e}_1,...,\mathbf{e}_{n+1})$ be an orthonormal
basis. We say that a tensor $f \in \otimes^m E^*$ is symmetric if
$f(v_1,...v_m) = f(v_{\tau(1)},...,v_{\tau(m)})$, for all $v_1,...,v_m \in
E$ and $\tau \in \mathfrak{S}_m$, the group of permutations of order $m$.
We denote by $\otimes^m_S T^*E$ the vector space of symmetric $m$-tensors
on $E$. There is a natural projection $\sigma : \otimes^m E^* \rightarrow
\otimes^m_S E^*$ given by
\[ \sigma\left(v_1^* \otimes ... \otimes v_m^*\right) = \dfrac{1}{m!} \sum_{\tau \in \mathfrak{S}_m} v_{\tau(1)}^* \otimes ... \otimes v_{\tau(m)}^*, \]
for all $v_1^*, ..., v_m^* \in E^*$. The metric $g$ induces a scalar
product $\langle \cdot,\cdot\rangle$ on $\otimes^m E^*$ by declaring the
basis $(\mathbf{e}^*_{i_1} \otimes ... \otimes \mathbf{e}^*_{i_m})_{1 \leq
i_1, ..., i_m \leq n+1}$ to be orthonormal which yields
\[ \langle u_1^*\otimes...\otimes u_m^*,v_1^*\otimes...\otimes v_m^*\rangle = \prod_{i=1}^m g^{-1}(u_i^*,v_i^*), \]
where $g^{-1}$ is the dual metric, that is the natural metric on $E^*$
which makes the musical isomorphism $\sharp : E \rightarrow E^*$ an
isometry. Since $\sigma$ is self-adjoint with respect to this metric, it is
an orthogonal projection. Let $(g_{ij})_{1\leq i,j \leq n+1}$ denote the
metric $g$ in the coordinates $(x_1,...,x_{n+1})$. Then the metric can be
expressed as
\[ \langle f,h\rangle = \sum_{i_1,...,i_m=1}^{n+1} f_{i_1 ... i_m} h^{i_1 ... i_m}, \]
where $h^{i_1 ... i_m} = \sum_{j_1,...,j_m=1}^{n+1} g^{i_1 j_1} ... g^{i_m j_m} h_{j_1 ... j_m}$. We
define the trace $\Tr_g : \otimes^m_S E^* \rightarrow \otimes^{m-2}_S E^*$
of a symmetric tensor by
\[ \Tr_g(f) = \sum_{i=1}^{n+1} f(\mathbf{e}_i,\mathbf{e}_i, \cdot, ..., \cdot). \]
In coordinates, $\Tr_g(f)(v_2,...,v_m) =
\Tr(g^{-1}f(\cdot,\cdot,v_2,...,v_m))$. Its adjoint with respect to the
scalar products is the map $I : \otimes^{m-2}_S E^* \rightarrow \otimes^m_S
E^*$ given by $I(u) = \sigma(g \otimes u)$.

Symmetric tensors can also be seen as homogeneous polynomials on the unit sphere of the
Euclidean space. We denote by $\Ss_E$ the $n$-dimensional unit sphere on
$(E,g)$ and by $dS$ the Riemannian measure on the sphere induced by the
metric $g|_{\Ss_E}$. We define $\pi_m : (x,v) \mapsto (x,\otimes^m v)$ for
$v \in E$; it induces a canonical morphism $\pi_m^* : \otimes_S^m E^*
\rightarrow C^\infty(\Ss_E)$ given by $\pi_m^* f (v) = f(v,...,v)$. Its
formal adjoint is $\langle \pi_m^* f, h \rangle_{L^2(\Ss_E,dS)} = \langle
f, {\pi_m}_*h\rangle_{\otimes^m T^*M}$, where $f \in \otimes_S^m T^*M, h
\in C^\infty(\Ss_E)$. In coordinates, \be \label{equation:pim-star}
({\pi_m}_* h)_{i_1...i_m} \coloneqq {\pi_m}_* h (\partial_{i_1}, ...,
\partial_{i_m}) = \sum_{j_1,...,j_m=1}^{n+1} g_{i_1j_1}...g_{i_mj_m}\int_{\Ss_E}h(v)v_{j_1}...v_{j_m}
dS \ee


Also remark that (\ref{equation:pim-star}) can be rewritten intrinsically
as \be \label{equation:pim-star-2} \forall u_1, ..., u_m \in E, \qquad
{\pi_m}_* h (u_1, ..., u_m) = \int_{\Ss_E} h(v) g(v,u_1) ... g(v,u_m) dv
\ee

The map ${\pi_m}_* \pi_m^*$ is an isomorphism we will study in the next
paragraph. Also note that $\pi_m^*(\sigma f) = \pi_m^*f$ (since all the
antisymmetric parts of the tensor $f$ vanish by plugging $m$ times the same
vector $v$).

We denote by $j_\xi$ the multiplication by $\xi$, that is $j_\xi : f
\mapsto \xi \otimes f$, and by $i_\xi$ the contraction, that is $i_\xi : f
\mapsto u(\xi^\sharp, \cdot, ..., \cdot)$. The adjoint of $i_\xi$ on
symmetric tensors with respect to the $L^2$-scalar product is $\sigma
j_\xi$, that is
\[ \forall f \in \otimes^{m-1}_S E^*, h \in \otimes^m_S E^*, \hspace{1cm} \langle \sigma j_\xi f, h \rangle = \langle f, i_\xi h \rangle \]

The space $\otimes^S_m E^*$ can thus be decomposed as the direct sum
\[ \otimes_S^m E^* = \ran \left(\sigma j_\xi|_{\otimes_S^{m-1} E^*}\right) \oplus^\bot \ker \left(i_\xi|_{\otimes_S^{m} E^*}\right) \]
We denote by $\pi_{\ker i_\xi}$ the projection onto the right space,
parallel to the left space. We will need the following
\begin{lemma}
\label{lemma:symbole-projection} For all $f, h \in \otimes^m_S E^*$,
\[ C_{n,m} \int_{\langle \xi, v\rangle=0} \pi_m^*f(v) \pi_m^*h(v) dS_\xi(v) =  \langle \pi_{\ker i_\xi} {\pi_m}_* \pi_m^*\pi_{\ker i_\xi} f, h\rangle ,\]
where
\[ C_{n,m} = \int_0^\pi \sin^{n-1+2m}(\varphi)d\varphi = \sqrt{\pi} \dfrac{\Gamma((n+2m)/2)}{\Gamma((n+1+2m)/2)}, \]
and $dS_\xi$ is the canonical measure induced on the $n-1$ dimensional
sphere $\Ss_{E,\xi}\coloneqq \Ss_E \cap \left\{\langle \xi,
v\rangle=0\right\}$.
\end{lemma}

\begin{proof}
We can write $h = \sigma j_\xi h_1 + h_2$ where $h_1 \in \otimes_S^{m-1}
E^*, h_2 \in \ker \left(i_\xi|_{\otimes_S^{m} T^*_xM}\right)$. Note that
$\pi_m^*(\sigma j_\xi h_1)(v) = \pi_m^*(j_\xi h_1)(v) = \langle \xi,v
\rangle \pi_{m-1}^*h_1(v)$ and this vanishes on $\left\{\langle \xi,
v\rangle=0\right\}$ (and the same holds for $f$). In other words, $\pi_m^*h
= \pi_m^*\pi_{\ker i_\xi}$ on $\left\{\langle \xi,v \rangle = 0\right\}$.
We are thus left to check that for $f, h \in \ker i_\xi$,
\[ C_{n,m} \int_{\langle \xi, v\rangle=0} \pi_m^*f(v) \pi_m^*h(v) dS_\xi(v) = \int_{\Ss_E} \pi_m^*f(v) \pi_m^*h(v) dS(v) \]
We will use the coordinates $v'=(v,\varphi) \in \Ss_{E,\xi} \times [0,\pi]$
on $\Ss_E$ which allow to decompose $v' = \sin(\varphi)v
+ \cos(\varphi)\xi^\sharp/|\xi|$. Then the measure on $\Ss_E$ disintegrates
as $dS = \sin^{n-1}(\varphi)d\varphi dS_\xi(v)$. Also remark that
$\pi_m^*f(v+\cos(\varphi)\xi^\sharp/|\xi|) = \pi_m^*f(v)$. Then, if
$C_{n,m} \coloneqq \int_0^\pi \sin^{n-1+2m}(\varphi) d\varphi$, we obtain:
\[ \begin{split} \int_{\langle \xi, v\rangle=0}  & \pi_m^*f(v) \pi_m^*h(v) dS_\xi(v) \\
& = C^{-1}_{n,m} \int_0^\pi \sin^{n-1+2m}(\varphi) d\varphi \int_{\langle \xi, v\rangle=0}  \pi_m^*f(v) \pi_m^*h(v) dS_\xi(v) \\
& = C^{-1}_{n,m} \int_0^\pi \int_{\langle \xi, v\rangle=0} \pi_m^*f(\sin(\varphi)v+\cos(\varphi)\xi^\sharp/|\xi|)  \\
& \hspace{3cm} \times \pi_m^*h(\sin(\varphi)v+\cos(\varphi)\xi^\sharp/|\xi|) \sin^{n-1}(\varphi) d\varphi dS_\xi(v) \\
& = C^{-1}_{n,m} \int_{\Ss_E} \pi_m^*f(v') \pi_m^*h(v') dS(v') \end{split} \]
\end{proof}

\subsubsection{Spherical harmonics}

\label{ssection:spherical-harmonics}

Let $\Delta|_{\Ss_E} \coloneqq \Div_{\Ss_E} \nabla_{\Ss_E}$ be the
Laplacian on the unit sphere $\Ss_E$ induced by the metric $g|_{\Ss_E}$ and
$\Delta$ be the usual Laplacian on $E$ induced by $g$. Let
\[ L^2(\Ss_E) = \bigoplus_{m=0}^{+\infty} \Omega_m \]
be the spectral break up in spherical harmonics, where $\Omega_m \coloneqq
\ker(\Delta|_{\Ss_E} + m(m+n-1))$ are the eigenspaces of the Laplacian. We
denote by $E_m$ the vector space of trace-free symmetric $m$-tensors, where
the trace is, as before, taken over the first two coordinates.

\begin{lemma}
$\pi_m^* : E_m \rightarrow \Omega_m$ is an isomorphism and ${\pi_m}_*
\pi_m^*|_{E_m} = \lambda_{m,n} \mathbbm{1}_{E_m}$, for some constant
$\lambda_{m,n} \neq 0$.
\end{lemma}

This also shows that, up to rescaling by the constant $\lambda_{m,n}$,
$\pi_m^* : E_m \rightarrow \Omega_m$ is an isometry. One could be more
accurate and actually show that the maps \be
\label{equation:decomposition-tenseurs} \pi_m^* : \otimes^m_S E^*
\rightarrow \oplus_{k=0}^{[m/2]} \Omega_{m-2k}, \qquad {\pi_m}_* : 
\oplus_{k=0}^{[m/2]} \Omega_{m-2k} \rightarrow \otimes^m_S E^* \ee are
isomorphisms, where $[m/2]$ stands for the integer part of $m/2$. This
follows from the (unique) decomposition of a symmetric tensor into a
trace-free part and a remainder (which lies in the image of the adjoint of
$\Tr$). More precisely, by iterating this process, one can decompose $u$ as
$u = \sum_{k=0}^{[m/2]} I^k(u_k)$, where $I : \otimes^{\bullet}_S E^*
\rightarrow \otimes^{\bullet+2}_S E^*$ is the adjoint of $\Tr$ with respect
to the scalar products and $u_k \in \otimes^{m-2k}_S E^*, \Tr(u_k)=0$ and
$\pi_m^*I^k(u_k) \in \Omega_{m-2k}$. Then
(\ref{equation:decomposition-tenseurs}) is an immediate consequence of the
previous lemma. The map ${\pi_m}_* \pi_m^*$ acts by scalar multiplication
on each component $I^k(u_k)$ (but with a different constant though, so
${\pi_m}_* \pi_m^*$ is not a multiple of the identity). Since we will only
need the fact that ${\pi_m}_* \pi_m^*$ is an isomorphism, we do not provide
further details.

\subsubsection{Symmetric tensors on a Riemannian manifold}

\textit{Decomposition in solenoidal and potential tensors.} We now consider
the Riemannian manifold $(M,g)$ and denote by $d\mu$ the Liouville measure
on the unit tangent bundle $SM$. All the previous definitions naturally
extend to the vector bundle $TM \rightarrow M$. For $f,h \in
C^\infty(M,\otimes_S^m T^*M)$, we define the $L^2$-scalar product
\[ \langle f,h \rangle = \int_M \langle f_x,h_x\rangle_x d\vol(x), \]
where $\langle \cdot,\cdot\rangle_x$ is the scalar product on $T_xM$
introduced in the previous paragraph and $d\vol$ is the Riemannian measure induced by $g$. The map $\pi_m^* :
C^\infty(M,\otimes_S^m T^*M) \rightarrow C^\infty(SM)$ is the canonical
morphism given by $\pi_m^*f(x,v) = f_x(v,...,v)$, whose formal adjoint with
respect to the two $L^2$-inner products (on $L^2(SM,d\mu)$ and
$L^2(\otimes_S^m T^*M,d\vol)$) is ${\pi_m}_*$, i.e., $\langle \pi_m^* f, h
\rangle_{L^2(SM,d\mu)} = \langle f, {\pi_m}_*h\rangle_{L^2(\otimes_S^m T^*M,
d\vol)}$.

If $\nabla$ denotes the Levi-Civita connection, we set $D \coloneqq \sigma
\circ \nabla : C^\infty(M,\otimes_S^m T^*M) \rightarrow
C^\infty(M,\otimes_S^{m+1} T^*M)$, the symmetrized covariant derivative. Its
formal adjoint with respect to the $L^2$-scalar product is
$D^*=-\Tr(\nabla)$ where the trace is taken with respect to the
two first indices, like in \ref{ssection:euclidien}. One has the following
relation between the geodesic vector field $X$ on $SM$ and the operator
$D$:

\begin{lemma}
$X\pi_m^* = \pi_{m+1}^*D$
\end{lemma}

The operator $D$ can be seen as a differential operator of order $1$. Its
principal symbol is given by $\sigma(D)(x,\xi) f \mapsto \sigma(\xi \otimes
f)=\sigma j_\xi f$ (see \cite[Theorem 3.3.2]{Sharafutdinov-94}).

\begin{lemma}
$D$ is elliptic. It is injective on tensors of odd order, and its kernel is
reduced to $\R \sigma(g^{\otimes m/2})$ on even tensors.
\end{lemma}

When $m$ is even, we will denote by $K_m=c_m \sigma(g^{\otimes m/2})$, with $c_m > 0$, a unitary vector in the kernel of $D$.

\begin{proof}
We fix $(x,\xi) \in T^*M$. For a tensor $u \in \otimes_s^m T^*_xM$, using
the fact that the antisymmetric part of $\xi \otimes u$ vanishes in the
integral:
\[ \langle \sigma(D)u, \sigma(D) u \rangle = \int_{\Ss_x^n} \langle\xi,v\rangle^2 {\pi_m^\ast u}^2(v) dS_x(v) = |\xi|^2 \int_{\Ss_x^n} \langle\xi/|\xi|,v\rangle^2 {\pi_m^\ast u}^2(v) dS_x(v) > 0, \]
unless $u \equiv 0$. Since $\otimes_s^m T^*_xM$ is finite dimensional, the
map
\[ (u,\xi/|\xi|) \mapsto \langle \sigma(D)(x,\xi/|\xi|)u, \sigma(D)(x,\xi/|\xi|) u \rangle,\]
defined on the compact set $\left\{u\in \otimes^m_S T^*_xM, |u|^2=1\right\}
\times \Ss^n$ is bounded and attains its lower bound $C^2 > 0$ (which is
independent of $x$). Thus $\|\sigma(x,\xi)\| \geq C|\xi|$, so the operator
is uniformly elliptic and can be inverted (on the left) modulo a compact
remainder: there exists pseudodifferential operators $Q, R$ of respective order
$-1,-\infty$ such that $QD= 1+R$.

As to the injectivity of $D$: if $Df = 0$ for some tensor $f \in
\mc{D}'(M,\otimes^m_S T^*M)$, then $f$ is smooth and $\pi_{m+1}^*Df =
X\pi_m^*f = 0$. By ergodicity of the geodesic flow, $\pi_m^*f=c \in
\Omega_0$ is constant. If $m$ is odd, then $\pi_m^*f(x,v)=-\pi_m^*f(x,-v)$
so $f \equiv 0$. If $m$ is even, then, by
\S\ref{ssection:spherical-harmonics}, $f = I^{m/2}(u_{m/2})$ where $u_{m/2}
\in \otimes^0_S E^*\simeq \R$ so $f=c' \sigma(g^{\otimes m/2})$.
\end{proof}

By classical elliptic theory, the ellipticity of $D$ implies that \be
\label{annexe:decomposition} H^{s}(M,\otimes^m_S T^*M) =
D(H^{s+1}(M,\otimes^{m-1}_S T^*M)) \oplus \ker D^*|_{H^{s}(M,\otimes^{m}_S
T^*M)}, \ee and the decomposition still holds in the smooth category and in
the $C^{k,\alpha}$-topology for $k \in \N, \alpha \in (0,1)$. This is the
content of the following theorem:

\begin{theorem}[Tensor decomposition]
\label{annexe:theoreme-decomposition} Let $s \in \R$ and $f \in
H^{s}(M,\otimes^m_S T^*M)$. Then, there exists a unique pair of symmetric
tensors $(p,h) \in H^{s+1}(M,\otimes^{m-1}_S T^*M) \times
H^{s}(M,\otimes^{m}_S T^*M)$ such that $f=Dp+h$ and $D^*h = 0$. Moreover, if $m=2l+1$ is odd, $\langle p, K_{2l}\rangle = 0$.
\end{theorem}

\textit{Tensorial distributions.} The spaces $H^s(M, \otimes^m_S T^*M)$
that have been mentioned so far are the $L^2$-based Sobolev spaces of order
$s \in \R$. They can be defined in coordinates (each coordinate of the
tensor has to be in $H^s_{\text{loc}}(\R)$) or more intrinsically by
setting $H^s(M, \otimes^m_S T^*M) \coloneqq
(\mathbbm{1}+D^*D)^{-s/2}L^2(M,\otimes^m_S T^*M)$. These two definitions are
equivalent by \cite[Proposition 7.3]{Shubin-01}, following the properties
of the operator $\mathbbm{1}+D^*D$ (it is elliptic, invertible,
positive). In the same fashion, the spaces $L^p(M, \otimes^m_S T^*M)$, for
$p \geq 1$ can be defined in coordinates. Note that the maps
\[
\pi_m^*: H^s(M, \otimes^m_S T^*M) \rightarrow H^s(SM) ,\quad {\pi_m}_* : H^s(SM)\rightarrow H^s(M, \otimes^m_S T^*M).
\]
are bounded for all $s \in \R$ (and they are bounded on $L^p$-spaces for $p
\geq 1$). The operator ${\pi_m}_*$ acts by duality on distributions,
namely:
\[ {\pi_m}_* : C^{-\infty}(SM) \rightarrow C^{-\infty}(M, \otimes^m_S T^*M), \qquad
\langle {\pi_m}_*f_1,f_2\rangle \coloneqq \langle f_1,\pi_m^*f_2\rangle \]
where $\langle\cdot,\cdot\rangle$ denotes the distributional pairing.  \\

\textit{The projection on solenoidal tensors.}
\label{sssection:projection-solenoidal-tensors} When $m$ is even, we denote by $\Pi_{K_m} := \langle K_m, \cdot \rangle K_m$ the orthogonal projection on $\R K_m$. We define $\Delta_m := D^*D + \eps(m)\Pi_{K_m}$, where $\eps(m) = 1$ for $m$ even, $\eps(m)=0$ for $m$ odd. The operator $\Delta_m$ is an
elliptic differential operator of order $2$ which is invertible: as a
consequence, its inverse is also pseudodifferential of order $-2$ (see \cite[Theorem 8.2]{Shubin-01}). We can
thus define the operator \be \label{equation:projection} \pi_{\ker D^*}
\coloneqq \mathbbm{1} - D\Delta_m^{-1}D^*. \ee One can check that this is
exactly the $L^2$-orthogonal projection on solenoidal tensors, it is a
pseudodifferential operator of order $0$ (as a composition of pseudodifferential operators).

Since $\sigma(D)(x,\xi) = \sigma j_\xi$, we know by
\S\ref{ssection:euclidien} that given $(x,\xi) \in T^*M$, the space
$\otimes_S^m T^*_xM$ breaks up as the direct sum
\[ \begin{split} \otimes_S^m T^*_xM & = \ran \left(\sigma(D)(x,\xi)|_{\otimes_S^{m-1} T^*_xM}\right) \oplus \ker \left(\sigma(D^*)(x,\xi)|_{\otimes_S^{m} T^*_xM}\right) \\
& = \ran \left(\sigma j_\xi|_{\otimes_S^{m-1} T^*_xM}\right) \oplus \ker \left(i_\xi|_{\otimes_S^{m} T^*_xM}\right) \end{split} \]
We recall that $\pi_{\ker i_\xi}$ is the projection on $\ker
\left(i_\xi|_{\otimes_S^{m} T^*_xM}\right)$ parallel to $\ran \left(\sigma
j_\xi|_{\otimes_S^{m-1} T^*_xM}\right)$.

\begin{lemma}
\label{lemma:projection-solenoidal} The principal symbol of $\pi_{\ker
D^*}$ is $\sigma_{\pi_{\ker D^*}}= \pi_{\ker i_\xi}$.
\end{lemma}

\begin{proof}
First, observe that:
\[ 
\begin{split}
D\Delta_m^{-1}D^*D\Delta_m^{-1}D^* & = D\Delta_m^{-1}(\Delta_m - \eps(m)\Pi_{K_m})\Delta_m^{-1}D^* \\
& = D \Delta_m^{-1} D^* - \eps(m) D \Delta_m^{-1} \Pi_{K_m} \Delta_m^{-1} D^*
\end{split}
\]
The second operator is smoothing so at the principal symbol level
\[ \sigma_{(D\Delta_m^{-1}D^*)^2} = \sigma_{D\Delta_m^{-1}D^*}^2 = \sigma_{D \Delta_m^{-1}D^*}, \]
which implies that $\sigma_{D\Delta_m^{-1}D^*}$ is a projection. Moreover, $\sigma_{D\Delta_m^{-1}D^*} = \sigma_D \sigma_{\Delta_m^{-1}} \sigma_{D^*}
= \sigma j_\xi \sigma_{\Delta_m^{-1}} i_\xi$, so it is the projection onto
$\ran \sigma j_\xi$ with kernel $\ker i_\xi$. Since $\pi_{\ker D^*} =
\mathbbm{1} - D\Delta_m^{-1}D^*$, the result is immediate.
\end{proof}

\section{On Livsic-type theorems}

\label{section:livsic}


We will denote by $d : \mc{M} \times \mc{M} \rightarrow \R$ the Riemannian
distance on $\mc{M}$ inherited from the Riemannian metric $g$. The
$\alpha$-Hölder norm of $f$ is defined by:
\[ \|f\|_{C^\alpha} \coloneqq \sup_{x \in \mc{M}} |f(x)| + \sup_{x,y \in \mc{M}, x \neq y} \dfrac{|f(x)-f(y)|}{d(x,y)^\alpha} = \|f\|_\infty + \|f\|_\alpha \]
In a series of inequalities, we will sometimes write $A \lesssim B$ to
denote the fact that there exists a universal constant $C > 0$ such that $A
\leq C \cdot B$. Note that a constant $C > 0$ may still appear from time to
time and, as usual, it may change from one line to another.

\subsection{Properties of Anosov flows}

We refer to the exhaustive \cite{Hasselblatt-Katok-95} and the forthcoming
book \cite{Fisher-Hasselblatt} for an introduction to hyperbolic dynamics.

\subsubsection{Stable and unstable manifolds}

The \textit{global stable} and \textit{unstable manifolds} $W^s(x), W^u(x)$
are defined by:
\[ \begin{array}{c} W^s(x) = \left\{ x' \in \mc{M}, d(\varphi_t(x), \varphi_t(x')) \rightarrow_{t \rightarrow +\infty} 0 \right\} \\
W^u(x) = \left\{ x' \in \mc{M}, d(\varphi_t(x), \varphi_t(x')) \rightarrow_{t \rightarrow -\infty}  0 \right\} \end{array}\]

For $\varepsilon > 0$ small enough, we define the \textit{local stable} and
\textit{unstable manifolds} $W^s_\varepsilon(x) \subset W^s(x),
W^u_\varepsilon(x) \subset W^u(x)$ by:
\[ \begin{array}{c} W^s_\varepsilon(x) = \left\{ x' \in W^s(x), \forall t \geq 0, d(\varphi_t(x), \varphi_t(x')) \leq \varepsilon \right\} \\
W^u_\varepsilon(x) = \left\{ x' \in W^u(z), \forall t \geq 0,
d(\varphi_{-t}(x), \varphi_{-t}(x')) \leq \varepsilon \right\} \end{array}\]
For all $\eps > 0$ small enough, there exists $t_0 > 0$ such that: \be
\label{eq:inc} \forall x \in \mc{M}, \forall t \geq t_0,
\varphi_t(W^s_{\eps}(x)) \subset W^s_{\eps}(\varphi_t(x)),
\varphi_{-t}(W^u_{\eps}(x)) \subset W^u_{\eps}(\varphi_{-t}(x)) \ee And:
\[ T_x W^s_{\eps}(x) = E_s(x), ~~~ T_x W^u_{\eps}(x) = E_u(x) \]

\subsubsection{Classical properties}

%
%
%

The main tool we will use to construct suitable periodic orbits is the
following classical shadowing property of Anosov flows. Part of the proof can
be found in \cite[Corollary 18.1.8]{Hasselblatt-Katok-95} and \cite[Theorem
5.3.2]{Fisher-Hasselblatt}. The last bound is a consequence of hyperbolicity
and can be found in \cite[Proposition 6.2.4]{Fisher-Hasselblatt}. For the
sake of simplicity, we will write $\gamma=[xy]$ if $\gamma$ is an orbit
segment with endpoints $x$ and $y$.
\begin{theorem}
\label{thm:shadowing} There exist $\eps_0>0$, $\theta>0$ and $C>0$ with the
following property. Consider $\eps<\eps_0$, and a finite or infinite sequence
of orbit segments $\gamma_i = [x_iy_i]$ of length $T_i$ greater than $1$ such
that for any $n$, $d(y_n,x_{n+1}) \leq \eps$. Then there exists a genuine
orbit $\gamma$ and times $\tau_i$ such that $\gamma$ restricted to $[\tau_i,
\tau_i+T_i]$ shadows $\gamma_i$ up to $C\eps$. More precisely, for all $t\in
[0, T_i]$, one has
\begin{equation}
\label{eq:d_hyperbolic}
  d(\gamma(\tau_i+t), \gamma_i(t)) \leq C \eps e^{-\theta \min(t, T_i-t)}.
\end{equation}
Moreover, $\abs{\tau_{i+1} - (\tau_i + T_i)} \leq C \eps$. Finally, if the
sequence of orbit segments $\gamma_i$ is periodic, then the orbit $\gamma$ is
periodic.
\end{theorem}
\begin{remark}
In this theorem, we could also allow the first orbit segment $\gamma_i$ to be
infinite on the left, and the last orbit segment $\gamma_j$ to be infinite on
the right. In this case,~\eqref{eq:d_hyperbolic} should be replaced by its
obvious reformulation: assuming that $\gamma_i$ is defined on $(-\infty, 0]$
and $\gamma_j$ on $[0,+\infty)$, we would get for some $\tilde\tau_{i+1}$
within $C\eps$ of $\tau_{i+1}$, and all $t\geq 0$
\begin{equation}
\label{eq:d_hyperbolic_neg}
  d(\gamma(\tilde\tau_{i+1}-t), \gamma_i(-t)) \leq C \eps e^{-\theta t}
\end{equation}
and
\begin{equation*}
  d(\gamma(\tau_{j}+t), \gamma_j(t)) \leq C \eps e^{-\theta t}.
\end{equation*}
\end{remark}

In particular, if $\gamma_0$ is an orbit segment $[xy]$ with $d(y, x)\leq
\eps_0$, then applying the above theorem to $\gamma_i \coloneqq \gamma_0$ for
all $i\in \Z$, one gets a periodic orbit that shadows $\gamma_0$: this is the
\emph{Anosov closing lemma}. We will also use thoroughly the version with two
orbit segments that are repeated to get a periodic orbit.

\subsubsection{Cover by parallelepipeds}

\label{ssect:paral}

We will now fix $\eps_0$ small enough so that the previous propositions are
guaranteed. For $\eps \leq \eps_0$, we define the set $W_\eps(x) \coloneqq \bigcup_{y \in W^{u}_\eps(x)} W^{s}_\eps(x)$. We can cover the manifold $\mc{M}$ by a finite union of flow boxes $\mc{U}_i := \cup_{t\in (-\delta,\delta)} \varphi_t(\Sigma_i)$, where $\Sigma_i := W_{\eps_0}(x_i)$ and $x_i \in \mc{M}$.

We denote by $\pi_i : \mc{U}_i \rightarrow \Sigma_i$ the projection by the flow
on the transverse section and we define $\mathfrak{t}_i : \mc{U}_i \rightarrow
\R$ such that $\pi_i(x) = \varphi_{\mathfrak{t}_i(x)}(x)$ for $x \in \mc{U}_i$.
We will need the following lemma:

\begin{lemma}
\label{lemma:holder-pit} $\pi_i, \mathfrak{t}_i$ are Hölder-continuous.
\end{lemma}

\begin{proof}
This is actually a general fact related to the Hölder regularity of the
foliation and the smoothness of the flow.

For the sake of simplicity, we drop the index $i$ in this proof. Let us
first prove the Hölder continuity for $x$ close to $\Sigma$ and $x'$ close
to $x$. We fix $p \in \Sigma$ and choose smooth local coordinates $\psi :
B(p,\eta) \rightarrow \R^{n+1} = \R \times \R^{n_s} \times \R^{n_u}$ around
$p$ (and centered at $0$) so that $d\psi_{p}(X) = \partial_{x_0}$. This
choice guarantees that in a neighborhood of $0$, the flow is transverse to
the hyperplane $\left\{0\right\}\times\R^{n_s+n_u}$. We still denote by
$\Sigma_\eta$ its image $\psi(\Sigma_\eta) \subset \R^{n+1}$, which is a
submanifold of Hölder regularity (the index $\eta$ indicates that we
consider the same objects intersected with the ball $B(x,\eta)$). Moreover,
there exists a Hölder-continuous homeomorphism $\Phi : S \rightarrow
\Sigma_\eta$, where $S \subset \left\{0\right\} \times\R^{n_s+n_u}$ (since
$\Sigma_\eta$ is a submanifold of $M$ with Hölder regularity). We consider
$\hat{\varphi} : (-\delta, \delta) \times S \rightarrow
\varphi_{(-\delta,\delta)}(S) =: V \supset \Sigma_\eta$ defined by
$\hat{\varphi}(t,z) = \varphi_t(0,z)$, which is a smooth diffeomorphism.
Remark that $\mathfrak{t}$ satisfies for $(0,z) \in S$,
$(\mathfrak{t}(z),z) = \hat{\varphi}^{-1}(\Phi(z))$. So it is
Hölder-continuous on $S$. Then $z \mapsto \pi(0,z) =
\varphi_{\mathfrak{t}(z)}(0,z)$ is Hölder-continuous on $S$ too.

We denote by $\pi_S : V \rightarrow S$ the projection and by
$\mathfrak{t}_S : V \rightarrow S$ the time such that $\pi_S(x) =
\varphi_{\mathfrak{t}_S(x)}(x)$. These two maps are smooth by the implicit
function theorem since the flow is transverse to $S$. Moreover, we have:
$\pi(x) = \pi|_S\left(\pi_S(x)\right)$ so $\pi$ is Hölder-continuous. And
$\mathfrak{t}(x) = \mathfrak{t}_S(x)+\mathfrak{t}|_S(\pi_S(x))$ so
$\mathfrak{t}$ is Hölder-continuous too. Note that by compactness of
$\Sigma$, this procedure can be done with only a finite number of charts,
thus ensuring the uniformity of the constants. Thus, $\pi_i,
\mathfrak{t}_i$ are Hölder-continuous in a neighborhood of $\Sigma$. Now,
in order to obtain the continuity on the whole cube $\mc{U}$, one can repeat the
same argument i.e., write the projection as the composition of a first
projection on a smooth small section $S$ defined in a neighborhood of
$\Sigma$ with the actual projection on $\Sigma$. This provides the sought
result.

\end{proof}

\subsection{Proof of the approximate Livsic theorem}

We now deal with the proof of Theorem~\ref{theorem:livsic-approche}.

\subsubsection{A key lemma.}

The following lemma states that we can find a sufficiently dense and yet
separated orbit in the manifold $\M$. The separation can only hold
transversally to the flow direction, and is defined as follows. Recall that $W_\eps(x) \coloneqq \bigcup_{y \in W^{u}_\eps(x)} W^{s}_\eps(x)$. Then we
say that a set $S$ is \emph{$\eps$-transversally separated} if, for any $x \in S$,
we have $S \cap W_\eps(x) = \{x\}$.

\begin{lemma}
\label{lemma:bonne-orbite} Consider a transitive Anosov flow on a compact
manifold. There exist $\beta_s, \beta_d>0$ such that the following holds. Let
$\eps > 0$ be small enough. There exists a periodic orbit $\mc{O}(x_0)
\coloneqq (\varphi_t x_0)_{0 \leq t \leq T}$ with $T \leq \eps^{-1/2}$ such
that this orbit is $\eps^{\beta_s}$-transversally separated and $(\varphi_t
x_0)_{0 \leq t \leq T-1}$ is $\eps^{\beta_d}$-dense. If $\kappa > 0$ is some
fixed constant, then one can also require that there exists a piece of
$\mc{O}(x_0)$ of length $\leq C(\kappa)$ which is $\kappa$-dense in the
manifold.
\end{lemma}
\begin{proof}
We could give a combinatorial construction in terms of Markov partitions and
carefully chosen sequences of symbols in the symbolic dynamics representation
of the flow. However, controlling rigorously the boundary effects on
separation is delicate. Instead, we give a geometric construction solely
based on the shadowing theorem. It is easy to obtain an
$\eps^{\beta_d}$-dense orbit by concatenating orbit segments thanks to the
shadowing Theorem. However, separation is harder to enforce. In this proof,
we introduce several constants, but none of them will depend on $\eps$.

Let us fix two periodic points $p_1$ and $p_2$ with different orbits
$\mc{O}(p_1)$ and $\mc{O}(p_2)$ of respective lengths $\ell_1$ and $\ell_2$.
By the shadowing theorem and transitivity, there exists an orbit $\gamma_-$
which is asymptotic to $\mc{O}(p_1)$ in negative time and to $\mc{O}(p_2)$ in
positive time. Also, there exists an orbit $\gamma_+$ which is asymptotic to
$\mc{O}(p_2)$ in negative time and to $\mc{O}(p_1)$ in positive time. On
$\gamma_-$, fix a point $z_0$, and $\rho_0 >0$ small enough so that $\gamma_-
\cup \gamma_+$ meets $W_{3\rho_0}(z_0)$ only at $z_0$, and $\mc{O}(p_1)$ and
$\mc{O}(p_2)$ are at distance $>3\rho_0$ of $z_0$. Denote by $C_0$ the
constant $C$ in the shadowing theorem~\ref{thm:shadowing}. Reducing $\rho_0$
if necessary, we can assume $\rho_0 < \eps_0$ where $\eps_0$ is given by
Theorem~\ref{thm:shadowing}. Let us also fix a large constant $C_1$, on which
our construction will depend.

We truncate $\gamma_-$ in positive time, stopping it at a fixed time where it
is within distance $\rho_0/(2C_0)$ of $p_2$, to get an orbit $\gamma'_-$. Let
$t_-$ be the largest time in $(-\infty, -2C_1\abs{\log\eps}]$ where
$\gamma'_-(t)$ is within distance $\eps$ of $p_1$. As the orbit $\gamma'_-$
converges exponentially quickly in negative time to $\mc{O}(p_1)$ by
hyperbolicity, one has $d(\gamma'_-(t), \mc{O}(p_1)) \leq \eps$ for $t \leq
-2C_1\abs{\log\eps}$, if $C_1$ is large enough. Hence, one needs to wait at
most $\ell_1$ before being $\eps$-close to $p_1$. This shows that the time
$t_-$ satisfies $t_- = -2C_1\abs{\log\eps}+O(1)$.

In the same way, we truncate $\gamma_+$ in negative time at a fixed time for
which it is within distance $\rho_0/(2C_0)$ of $p_2$, obtaining an orbit
$\gamma'_+$. We denote by $t_+$ the smallest time in $[2C_1\abs{\log\eps},
+\infty)$ with $d(\gamma'_+(t), p_2) \leq \eps$. It satisfies $t_+ =
2C_1\abs{\log\eps} + O(1)$.

As the flow is transitive, it has a dense orbit. Therefore, for any $x, y$,
there exists an orbit $\gamma_{x,y}$ starting from a point within distance
$\rho_0/(2C_0)$ of $x$, ending at a point within distance $\rho_0/(2C_0)$ of
$y$, and with length $\in [1, T_0]$ where $T_0$ is fixed and independent of
$x$ and $y$.

To any $x$, we associate an orbit as follows. Start with $\gamma'_-$, then
follow $\gamma_{p_2, \varphi_{- C_1\abs{\log \eps}} x}$, then follow the
orbit of $x$ between times $-C_1\abs{\log \eps}$ and $C_1\abs{\log \eps}$,
then follow $\gamma_{\varphi_{C_1\abs{\log \eps}}x, p_2}$, then follow
$\gamma'_+$. In this sequence, the distance between an endpoint of a piece
and the starting point of the next one is always less than $\rho_0/C_0$.
Hence, Theorem~\ref{thm:shadowing} applies and yields an infinite orbit
$\gamma'_x$, that follows the above pieces of orbits up to $\rho_0$. If $C_1$
is large enough,~\eqref{eq:d_hyperbolic} implies that $x$ is within distance
at most $\eps$ of $\gamma'_x$. The inequality~\eqref{eq:d_hyperbolic_neg}
shows that $\gamma'_-(t_-)$ and the corresponding point $x_-$ on $\gamma'_x$
are within distance $e^{-\theta t_-}$. If $C_1$ is large enough, this is
bounded by $\eps$ since $t_- = -2C_1 \abs{\log\eps} + O(1)$. Therefore,
$d(x_-, p_1) \leq 2 \eps$. In the same way, the point $x_+$ on $\gamma'_x$
matching $\gamma'_+(t_+)$ is within distance $\eps$ of $\gamma'_+(t_+)$, and
therefore within distance $2\eps$ of $p_1$. Let us truncate $\gamma'_x$
between $x_-$ and $x_+$, to get an orbit segment $\gamma_x$ of length
$6C_1\abs{\log \eps} + O(1)$, starting and ending within $2\eps$ of $p_1$,
and passing within $\eps$ of $x$.

Let $\beta_d = 1/(3\dim(\mc{M}))$. We define a sequence of points of $\mc{M}$
as follows. Let $x_1$ be an arbitrary point for which the
$C(\kappa)$-beginning of its orbit is $\kappa/2$-dense, to guarantee in the
end that the last condition of the lemma is satisfied. If $\gamma_{x_1}$ is
not $\eps^{\beta_d}/2$-dense, we choose another point $x_2$ which is not in
the $\eps^{\beta_2}/2$-neighborhood of $\gamma_{x_1}$. Then $\gamma_{x_1}\cup
\gamma_{x_2}$ contain both $x_1$ and $x_2$ in their $\eps$-neighborhood, and
therefore in their $\eps^{\beta_d}/2$-neighborhood. If $\gamma_{x_1}\cup
\gamma_{x_2}$ is still not $\eps^{\beta_d}/2$-dense, then we add a third
piece of orbit $\gamma_{x_3}$, and so on. By compactness, this process stops
after finitely many steps, giving a finite sequence $x_1,\dotsc, x_N$.

As all $\gamma_{x_i}$ start and end with $p_1$ up to $2\eps$, we can glue the
sequence
\begin{equation*}
  \dotsc, \gamma_{x_N},\gamma_{x_1}, \gamma_{x_2},\dotsc, \gamma_{x_N}, \gamma_{x_1},\dotsc
\end{equation*}
thanks to Theorem~\ref{thm:shadowing}. We get a periodic orbit $\gamma$ which
shadows them up to $2C_0 \eps$. We claim this orbit satisfies the
requirements of the lemma. We should check its length, its density, and its
separation.

Let us start with the length. The points $x_i$ are separated by at least
$\eps^{\beta_d}/3$. The balls of radius $\eps^{\beta_d}/6$ are disjoint, and
each has a volume $\geq c \eps^{\beta_d \cdot \dim(\mc{M})} = c \eps^{1/3}$.
We get that the number $N$ of points $x_i$ is bounded by $C \eps^{-1/3}$. As
each piece $\gamma_{x_i}$ has length at most $C \abs{\log \eps}$, it follows
that the total length of $\gamma$ is bounded by $C \abs{\log \eps}
\eps^{-1/3} \leq \eps^{-1/2}$.

Let us check the density. By construction, the union of the $\gamma_{x_i}$ is
$\eps^{\beta_d}/2$-dense. As $\gamma$ approximates each $\gamma_{x_i}$ within
$2C_0\eps$, it follows that $\gamma$ is $2C_0\eps+\eps^{\beta_d}/2$ dense,
and therefore $\eps^{\beta_d}$-dense. In the statement of the lemma, we
require the slightly stronger statement that if one removes a length $1$
piece at the end of the orbit it remains $\eps^{\beta_d}$-dense. Such a
length $1$ piece in $\gamma_{x_N}$ consists of points that are within $2\eps$
of $\mc{O}(p_1)$. They are approximated within $\eps^{\beta_d}$ by the start
and end of all the other $\gamma_{x_i}$.

Finally, let us check the more delicate separation, which has motivated the
finer details of the construction as we will see now. Let $\beta_s$ be
suitably large. We want to show that any two points $x,y$ of $\gamma$ within
distance $\eps^{\beta_s}$ are on the same local flow line. Since the
expansion of the flow is at most exponential, for any $t \leq 20C_1 \abs{\log
\eps}$, we have $d(\varphi_t x, \varphi_t y) \leq \eps$ if $\beta_s$ is large
enough. In the piece of $\gamma$ of length $10C_1 \abs{\log\eps}$ starting at
$x$, there is an interval $[t_1, t_2]$ of length $4C_1\abs{\log\eps}+O(1)$
during which $\varphi_t x$ is within distance at most $\rho_0/2$ of
$\mc{O}(p_1)$, corresponding to the junction between the orbits
$\gamma_{x_{i}}$ and $\gamma_{x_{i+1}}$ where $i$ is such that $x$ belongs to
the shadow of $\gamma_{x_{i-1}}$. For $t \in [t_1, t_2]$, one also has
$d(\varphi_t y, \mc{O}(p_1)) \leq \rho_0$ as the orbits follow each other up
to $\eps$. Note that in each $\gamma_j$ the consecutive time spent close to
$\mc{O}(p_1)$ is bounded by $2C_1\abs{\log\eps}$ as we have forced a passage
close to $p_2$ (and therefore far away from $\mc{O}(p_1))$) after this time
in the construction. It follows that also for $y$ the time interval
$[t_1,t_2]$ has to correspond to a junction between two orbits
$\gamma_{x_{j}}$ and $\gamma_{x_{j+1}}$.
Consider the smallest times $t$ and $t'$ after the junctions for which
$\varphi_t(x)$ and $\varphi_{t'}(y)$ are $2\rho_0$-close to $z_0$. Since the
orbit $\gamma'_-$ meets $W_{3\rho_0}(z_0)$ at the single point $z_0$, these
times have to correspond to each other, i.e., the orbits are synchronized up
to an error $O(\eps)$. To conclude, it remains to show that $i=j$. Suppose by
contradiction $i<j$ for instance. The orbit of $x$ follows $\gamma_{x_i}$ up
to $2C_0 \eps$, the orbit of $y$ follows $\gamma_{x_j}$ up to $2C_0 \eps$,
and the orbits of $x$ and $y$ are within $\eps$ of each other. We deduce that
$\gamma_{x_i}$ and $\gamma_{x_j}$ follow each other up to $(4C_0+1)\eps$.
Since $x_j$ is within $\eps$ of $\gamma_{x_j}$, it follows that $x_j$ is at
within $(4C_0+2)\eps$ of $\gamma_{x_i}$. This is a contradiction with the
construction, as we could have added the point $x_j$ only it was not in the
$\eps^{\beta_s}$-neighborhood of $\gamma_{x_i}$, and $\eps^{\beta_s} >
(4C_0+2)\eps$ if $\eps$ is small enough.
\end{proof}

\subsubsection{Construction of the approximate coboundary.}

Let us now prove Theorem~\ref{theorem:livsic-approche}. The result is obvious
if $\eps$ is bounded away from $0$, by taking $u=0$ and $h=f$. Hence, we can
assume that $\eps$ is small enough to apply Lemma~\ref{lemma:bonne-orbite},
with $\kappa=\eps_0$. On the orbit $\mc{O}(x_0)$ given by this lemma, we
define a function $\tilde{u}$ by $\tilde{u}(\varphi_t x_0) = \int_0^t
f(\varphi_s x_0) ds$. Note that it may not be continuous at $x_0$. As a
consequence, we will rather denote by $\mc{O}(x_0)$ the set $(\varphi_t
x_0)_{0 \leq t \leq T-1}$ (which satisfies the required properties of density
and transversality) in order to avoid problems of discontinuity.

\begin{lemma}
There exist $\beta_1,C > 0$ independent of $\eps$ such that
$\|\tilde{u}\|_{C^{\beta_1}(\mc{O}(x_0))} \leq C$.
\end{lemma}
\begin{proof}
We first study the Hölder regularity of $\tilde{u}$, namely we want to
control $|\tilde{u}(x)-\tilde{u}(y)|$ by $C d(x,y)^{\beta_1}$ for some
well-chosen exponent $\beta_1$, when $d(x,y) \leq \eps_0$ (where $\eps_0$ is
the scale under which the shadowing theorem~\ref{thm:shadowing} holds). If
$x$ and $y$ are on the same local flow line, then the result is obvious since
$f$ is bounded by $1$, so we are left to prove that $\tilde{u}$ is
transversally $C^{\beta_1}$. Consider $x=\varphi_{t_0} x_0 \in \mc{O}(x_0)$
and $y=\varphi_{t_0+t} \in W_{\eps_0}(x)$. By transversal separation of
$\mc{O}(x_0)$, these points satisfy $d(x,y)\geq \eps^{\beta_s}$. We can close
the segment $[xy]$ i.e., we can find a periodic point $p$ such that $d(p,x)
\leq C d(x,y)$ with period $t_p = t + \tau$, where $|\tau| \leq Cd(x,y)$
which shadows the segment. Then:
\[
|\tilde{u}(y)-\tilde{u}(x)| \leq \underbrace{\left|\int_0^t f(\varphi_s x) ds - \int_0^{t_p} f(\varphi_s p) ds\right|}_{=\text{(I)}} + \underbrace{\left|\int_0^{t_p} f(\varphi_s p) ds \right|}_{=\text{(II)}}
\]
The first term $\text{(I)}$ is
bounded by $Cd(x,y)^{\beta_1'}$ for some $\beta'_1 > 0$ depending on the
dynamics, whereas the second term $\text{(II)}$ is bounded — by assumption
— by $\eps t_p$. But $\eps t_p \lesssim \eps t \lesssim  \eps T \lesssim
\eps^{1/2} \lesssim d(x,y)^{1/2\beta_s}$. We thus obtain the sought result
with $\beta_1 \coloneqq \min(\beta'_1,1/2\beta_s)$.

We now prove that $\tilde{u}$ is bounded for the $C^0$-norm. We know that
there exists a segment of the orbit $\mc{O}(x_0)$ — call it $S$ — of length
$\leq C$ which is $\eps_0$-dense in $\M$. In particular, for any $x \in
\mc{O}(x_0)$, there exists $x_S \in S$ with $d(x,x_S) \leq \eps_0$, and
therefore $\abs{\tilde u(x)-\tilde u(x_S)} \leq C d(x,x_S)^{\beta_1} \leq C
\eps_0^{\beta_1}$ thanks to the Hölder control of the previous paragraph.
Using the same argument with $x_0$, we get as $\tilde u(x_0)=0$
\begin{equation*}
  \abs{\tilde u(x)} = \abs{\tilde u(x)- \tilde u(x_0)} \leq \abs{\tilde u(x) - \tilde u(x_S)}
  + \abs{\tilde u(x_S) - \tilde u((x_0)_S)} + \abs{\tilde u(x_0)- \tilde u((x_0)_S)}.
\end{equation*}
The first and last term are bounded by $C \eps_0^{\beta_1}$, and the middle
one is bounded by $C$ as $S$ has a bounded length and $\|f\|_{C^0} \leq 1$.
\end{proof}

For each $i$, we extend the function $\tilde u$ (defined on $\mc{O}(x_0)$) to
a Hölder function $u_i$ on $\Sigma_i$, by the formula $u_i(x) = \sup \tilde
u(y) - \norm{\tilde{u}}_{C^{\beta_1}(\mc{O}(x_0))} d(x,y)^{\beta_1}$, where
the supremum is taken over all $y \in \mc{O}(x_0)$. With this formula, it is
classical that the extension is Hölder continuous, with
$\norm{u_i}_{C^{\beta_1}(\Sigma_i)} \leq
\|\tilde{u}\|_{C^{\beta_1}(\mc{O}(x_0))}$. We then push the function $u_i$ by
the flow in order to define it on $\mc{U}_i$ by setting for $x \in \Sigma_i,
\varphi_tx \in \mc{U}_i$: $u_i(\varphi_t x) = u_i(x) + \int_0^{t} f(\varphi_s
x) ds$. Note that by Lemma \ref{lemma:holder-pit}, the extension is still
Hölder with the same regularity. We now set $u \coloneqq \sum_i u_i \theta_i$
and $h\coloneqq f-Xu=-\sum_i u_i X\theta_i$. The functions $X\theta_i$ are
uniformly bounded in $C^\infty$, independently of $\eps$ so the functions
$u_i X\theta_i$ are in $C^{\beta_1}$ with a Hölder norm independent of $\eps
> 0$ and thus $\|h\|_{C^{\beta_1}} \leq C$.

\begin{lemma}
$\|h\|_{C^{\beta_1/2}} \leq \eps^{\beta_3/2}$
\end{lemma}

\begin{proof}
We claim that $h$ vanishes on $\mc{O}(x_0)$: indeed, on $\mc{U}_i \cap
\mc{O}(x_0)$ one has $u_i \equiv \tilde{u}$ and thus $h = -\tilde{u} \sum_i
X\theta_i = -\tilde{u} X \sum_i \theta_i = -\tilde{u} X \mathbf{1} = 0$.
Since $\mc{O}(x_0)$ is $\eps^{\beta_d}$-dense and $\|h\|_{C^{\beta_1}} \leq
C$, we get that $\|h\|_{C^0} \leq C \eps^{\beta_1 \beta_d} = C
\eps^{\beta_3}$, where $\beta_3 = \beta_1\beta_d$. By interpolation, we
eventually obtain that $\|h\|_{C^{\beta_1/2}} \leq
\eps^{\beta_3/2}$.\end{proof}

\begin{proof}[Proof of Theorem \ref{theorem:livsic-approche}]
The previous lemma provides the sought estimate on the remainder $h$. This
completes the proof of Theorem \ref{theorem:livsic-approche}.
\end{proof}

\section{Generalized geodesic X-ray transform}

\label{section:pi}

From now on, we will rather use the dual decomposition of the cotangent
space $T^*\M = E_0^* \oplus E_u^* \oplus E_s^*$, where $E_0^*(E_u \oplus
E_s) = 0, E_s^*(E_s \oplus \R X)=0, E_u^*(E_u \oplus \R X) = 0$. If
${A}^{-\top}$ denotes the inverse transpose of a linear operator $A$, then
the dual estimates to (\ref{equation:hyperbolicite}) are:
\be \begin{array}{c} |d\varphi^{-\top}_t(x) \cdot \xi|_{\varphi_t(x)} \leq C e^{-\lambda t} |\xi|_{x}, ~~ \forall t > 0, \xi \in E^*_s(x) \\
|d\varphi_t(x) \cdot \xi|_{\varphi_t(x)} \leq C e^{-\lambda |t|} |\xi|_{x},
~~ \forall t < 0, \xi \in E^*_u(x),\end{array}, \ee where $|\cdot|_x$ is
now $g^{-1}$, the dual metric to $g$ (which makes the musical isomorphism
$\flat : T\M \rightarrow T^*\M$ an isometry). For the sake of simplicity, we now assume that $X$ generates a contact Anosov flow; the results of this paragraph will be applied to the case of an Anosov geodesic flow. It would actually be sufficient to assume that the flow is Anosov, preserves a smooth measure and that it is  mixing for this measure. Note that a contact Anosov flow is exponentially mixing by \cite{Liverani-04}. We will denote by $\mu$ the \emph{normalized} volume form induced by the contact $1$-form. In the case of a geodesic flow, $\mu$ is nothing but the Liouville volume form. By $L^2(\mc{M})$, we will always refer to the space $L^2(\mc{M},d\mu)$. The orthogonal projection on the constant function is denoted by $\mathbf{1}\otimes\mathbf{1}$.

\subsection{Resolvent of the flow at \texorpdfstring{$0$}{0}}

By \cite{Faure-Sjostrand-11}, we know that the resolvents
$R_\pm(\lambda)\coloneqq (X \pm \lambda)^{-1} : \mc{H}_\pm^s \rightarrow
\mc{H}_\pm^s$ (initially defined for $\Re(\lambda) > 0$) admit a meromorphic extension to the half-space
$\left\{\Re(\lambda)>-cs\right\}$ — where $\mc{H}_\pm^s$ are
\emph{anisotropic Sobolev spaces} — and thus $R_\pm(\lambda) : C^\infty(\M)
\rightarrow \mc{D}'(\M)$ admit a meromorphic extension to the whole complex
plane. For $\Re(\lambda) > 0$, $R_\pm(\lambda) : L^2(\mc{M}) \rightarrow
L^2(\mc{M})$ are bounded and the expression $R_+(\lambda)$ is given by \be
\label{equation:r_plus} R_+(\lambda) = (X+\lambda)^{-1} = \int_0^{+\infty}
e^{-\lambda t} e^{-tX} dt, \ee where $e^{-tX}f(x) = f(\varphi_{-t}(x))$ for
$f \in C^\infty(\M), x \in \M$.

In a neighborhood of $0$, we can thus write the Laurent expansions \be
\label{equation:laurent-expansion} R_+(\lambda) = R_0^+ + \dfrac{\mathbf{1}
\otimes \mathbf{1}}{\lambda} + \mc{O}(\lambda), \qquad R_-(\lambda) = R_0^-
- \dfrac{\mathbf{1}\otimes \mathbf{1}}{\lambda} + \mc{O}(\lambda), \ee
where $R_0^+ : \mc{H}_+^s \rightarrow \mc{H}_+^s, R_0^- : \mc{H}_-^s
\rightarrow \mc{H}_-^s$ are bounded. Since $H^s \subset \mc{H}_\pm^s
\subset H^{-s}$, we obtain that $R_0^\pm : H^s \rightarrow H^{-s}$ are
bounded and thus $(R_0^+)^* : H^s \rightarrow H^{-s}$ is bounded too.
Moreover, it is easy to check that formally $(R_0^+)^*=-R_0^-$ (i.e., the
operators coincide on $C^\infty(\M)$), in the sense that for all $f_1,f_2
\in C^\infty(\M)$, $\langle R_0^- f_1,f_2 \rangle_{L^2(\M)} = \langle f_1,
- R_0^+ f_2 \rangle_{L^2(\M)}$. Since $C^\infty(\M)$ is dense in $H^s(\M)$,
we obtain that $(R_0^+)^* = -R_0^-$ on $H^s(\M)$, in the sense that for all
$f_1,f_2 \in H^s(\M)$, $\langle R_0^- f_1,f_2 \rangle_{L^2(\M)} = \langle
f_1, -R_0^+ f_2 \rangle_{L^2(\M)}$.

Also remark that, as operators $C^\infty(\M) \rightarrow \mc{D}'(\M)$, one
has: \be \label{equation:identites} XR_0^+ = R_0^+X = \mathbbm{1} -
\mathbf{1} \otimes \mathbf{1}, \quad XR_0^- = R_0^- X = \mathbbm{1} -
\mathbf{1} \otimes \mathbf{1} \ee

For the sake of simplicity, we will write $R_0\coloneqq R_0^+$. We introduce the operator \be
\Pi \coloneqq R_0 + R_0^*, \ee the sum of the two holomorphic parts of the
resolvent. An easy computation, using (\ref{equation:laurent-expansion}),
proves that $\Pi(\mathbf{1}) = 0$ and the image $\Pi(C^\infty(\M))$ is
orthogonal to the constants. We recall the

\begin{theorem}{\cite[Theorem 1.1]{Guillarmou-17-1}}
\label{theorem:guillarmou} For all $s > 0$, the operator $\Pi : H^s(\M) \rightarrow
H^{-s}(\M)$ is bounded, selfadjoint and satisfies:
\begin{enumerate}
\item $\forall f \in H^s(\M), X \Pi f = 0, $
\item $\forall f \in H^{s}(\M)$ such that $Xf\in H^s(\M)$, $\Pi X f = 0.$
\end{enumerate}
If $f \in H^s(\M)$ with $\langle f,\mathbf{1}\rangle_{L^2} = 0$, then $f
\in \ke \Pi$ if and only if there exists a solution $u \in H^s(\M)$ to the
cohomological equation $X u = f$, and $u$ is unique modulo constants.
\end{theorem}

There exists two other characterizations of the operator $\Pi$ that are
more tractable and which we detail in the next proposition. We set
$\Pi_{\lambda} \coloneqq \mathbf{1}_{(-\infty,\lambda]}(-iX)$.

\begin{proposition}
\label{proposition:caracterisation}
For $f_1,f_2 \in C^\infty(\M)$ such that $\langle f,\mathbf{1}\rangle_{L^2} = 0$:
\begin{enumerate}
\item $\langle \Pi f_1, f_2 \rangle = 2\pi \partial_\lambda|_{\lambda =
    0} \langle \Pi_\lambda f_1, f_2 \rangle, $
\item $\langle \Pi f_1, f_2 \rangle = \int_{-\infty}^{+\infty} \langle
    f_1 \circ \varphi_t , f_2 \rangle dt.$
\end{enumerate}
\end{proposition}

\begin{proof}
(1) For $f_1,f_2 \in C^\infty(\M)$ such that $\int_{\M} f_i d\mu = 0$,
we have using Stone's formula, for $\delta > 0$:
\[ \begin{split} \langle \Pi_{\lambda+\delta} f_1, f_2 \rangle - \langle \Pi_{\lambda-\delta} f_1, f_2 \rangle & = \langle \mathbf{1}_{[\lambda-\delta,\lambda+\delta]} f_1,f_2 \rangle \\
& = \dfrac{1}{2\pi} \int_{\lambda-\delta}^{\lambda+\delta}  \langle (R_+(-i\lambda)-R_-(i\lambda))f_1,f_2\rangle d\lambda \end{split} \]
Dividing by $2\delta$ and passing to the limit $\delta \rightarrow 0^+$, we
obtain $\partial_\lambda|_{\lambda = 0} \langle \Pi_\lambda f_1, f_2
\rangle = \frac{1}{2\pi} \langle (R_0^+-R_0^-)f_1,f_2\rangle
=\frac{1}{2\pi} \langle \Pi f_1,f_2\rangle $.

(2) Thanks to the exponential decay of correlations (see
\cite{Liverani-04}), one can apply Lebesgue's dominated convergence theorem
in the limit $\lambda \rightarrow 0^+$ in the following expression
\[ \langle \Pi f_1,f_2\rangle = \lim_{\lambda \rightarrow 0^+} \int_{-\infty}^{+\infty} e^{-\lambda |t|} \langle f_1 \circ \varphi_{-t} , f_2 \rangle dt, \]
and the result is then immediate. Note that a polynomial decay would have been sufficient.
\end{proof}

The quantity $\langle \Pi f ,f \rangle$ is sometimes referred to in the
literature as the \emph{variance} of the flow. In particular, it enjoys the
following positivity property:

\begin{lemma}
\label{lemma:positivite-operateur-pi} The operator $\Pi : H^s(\M)
\rightarrow H^{-s}(\M)$ is positive in the sense of quadratic forms, namely
$\langle \Pi f,f \rangle_{L^2} \geq 0$ for all real-valued $f \in H^s(\M)$.
\end{lemma}

There are different ways of proving this lemma, related to the different
characterizations of the operator $\Pi$. We only detail one of them which
is in the dynamical spirit of this article. Another way could be to use the first item of Proposition \ref{proposition:caracterisation} and the fact that the spectral measure $\Pi_\lambda$ is non-decreasing.

\begin{proof}
By density, it is sufficient to prove the lemma for a real-valued $f \in
C^\infty(\M)$. We will actually show that for $\lambda > 0$:
\[ \langle \left( R_+(\lambda) - \dfrac{\mathbf{1} \otimes \mathbf{1}}{\lambda}  \right) f, f \rangle = \langle R_+(\lambda) f, f \rangle - \frac{1}{\lambda} \left(\int_{\M} f d\mu\right)^2\geq 0 \]
The same arguments being valid for $R_-(\lambda)$, we will deduce the
result by taking the limit $\lambda \rightarrow 0^+$. By Parry' formula
\cite[Paragraph 3]{Parry-88}, we know that:
\be \label{equation:parry}
\langle R_+(\lambda) f, f \rangle = \lim_{T \rightarrow \infty}
\dfrac{1}{N(T)} \sum_{\ell(\gamma) \leq T} e^{\int_\gamma J^u}
\dfrac{1}{\ell(\gamma)} \int_0^{\ell(\gamma)} R_+(\lambda)f(\varphi_t z)
f(\varphi_tz) dt, \ee where $\gamma$ is a periodic orbit, $z \in \gamma$, $\ell(\gamma)$ is the length of $\gamma$ and $N(T)= \sum_{\ell(\gamma)
\leq T}e^{\int_\gamma J^u}$ is a normalizing coefficient,
\[ J^u : x \mapsto \partial_t \det d\varphi_t(x)|_{E_u(x)}|_{t=0}\]
is the unstable Jacobian (or the \emph{geometric potential}). Let us fix a
closed orbit $\gamma$ and a base point $z \in \gamma$. We set $\tilde{f}(t)
\coloneqq f(\varphi_tz)$ which we see as a smooth function, $\ell$-periodic
on $\R$ (with $\ell \coloneqq \ell(\gamma)$). Since $R_+(\lambda)$ commutes
with $X$, $R_+(\lambda)$ acts as a Fourier multiplier on functions defined
on $\gamma$. As a consequence, if we decompose $\tilde{f}(t) = \sum_{n \in
\Z} c_n e^{2i\pi nt/\ell}$, we have:
\[ \begin{split} R_+(\lambda)\tilde{f} (t) & = \int_0^{+\infty} e^{-\lambda s} \tilde{f}(t+s) ds \\
& = \sum_{n \in \Z} c_n e^{2i\pi n t/\ell} \int_0^{+\infty} e^{-(\lambda-2 i \pi n / \ell)s} ds \\
& = \sum_{n \in \Z} \dfrac{c_n (\lambda +2i \pi n /\ell)}{\lambda^2 +4 \pi^2 n^2 /\ell^2} e^{2i\pi n t/\ell} \end{split} \]
Then:
\[ \langle R_+(\lambda)\tilde{f} , \tilde{f} \rangle_{L^2}  = \dfrac{1}{\ell} \int_0^{\ell} R_+(\lambda)\tilde{f}(t) \tilde{f}(t) dt  = \sum_{n \in \Z} \dfrac{|c_n|^2  (\lambda + 2i \pi n/ \ell)}{\lambda^2 + 4\pi^2 n^2/\ell^2} = \lambda  \sum_{n \in \Z} \dfrac{|c_n|^2}{\lambda^2 + 4\pi^2 n^2/\ell^2},  \]
by oddness of the imaginary part of the sum. In particular: \be
\label{equation:fourier} \dfrac{1}{\ell} \int_0^{\ell}
R_+(\lambda)\tilde{f}(t) \tilde{f}(t) dt \geq \dfrac{|c_0|^2}{\lambda} =
\dfrac{1}{\lambda} \left(\dfrac{1}{\ell} \int_0^\ell \tilde{f}(t) dt
\right)^2\ee Inserting (\ref{equation:fourier}) into
(\ref{equation:parry}), then applying Jensen's convexity inequality:
\[ \begin{split} \langle R_+(\lambda) f, f \rangle & \geq \lambda^{-1} \lim_{T \rightarrow \infty} \dfrac{1}{N(T)} \sum_{\ell(\gamma) \leq T} e^{\int_\gamma J^u} \left( \dfrac{1}{\ell(\gamma)} \int_0^{\ell(\gamma)}  f(\varphi_tz) dt \right)^2  \\
& \geq \lambda^{-1}  \lim_{T \rightarrow \infty} \left(\dfrac{1}{N(T)} \sum_{\ell(\gamma) \leq T} e^{\int_\gamma J^u} \dfrac{1}{\ell(\gamma)} \int_0^{\ell(\gamma)}  f(\varphi_tz) dt \right)^2 = \dfrac{1}{\lambda}\left(\int_{SM} f d\mu\right)^2, \end{split} \]
where we used again Parry's formula in the last equality.
\end{proof}

\subsection{The normal operator}

We now consider a smooth closed manifold $(M,g)$ with Anosov geodesic flow and define $\mc{M} := SM$, the unit tangent bundle (with respect to the metric $g$). We introduce \be \Pi_m \coloneqq {\pi_m}_* (\Pi + \mathbf{1}\otimes
\mathbf{1}) \pi_m^* \ee Recall from
\S\ref{sssection:projection-solenoidal-tensors} that given $(x,\xi) \in
T^*M$, the space $\otimes_S^m T^*_xM$ decomposes as the direct sum
\[ \begin{split} \otimes_S^m T^*_xM & = \ran \left(\sigma_D(x,\xi)|_{\otimes_S^{m-1} T^*_xM}\right) \oplus \ker \left(\sigma_{D^*}(x,\xi)|_{\otimes_S^{m} T^*_xM}\right) \\
& = \ran \left(\sigma j_\xi|_{\otimes_S^{m-1} T^*_xM}\right) \oplus \ker \left(i_\xi|_{\otimes_S^{m} T^*_xM}\right) \end{split} \]
The projection on the right space parallel to the left space is denoted by
$\pi_{\ker i_\xi}$ and $\Op(\pi_{\ker i_\xi}) = \pi_{\ker D^*} +
S$ by Lemma \ref{lemma:projection-solenoidal}, where $S \in \Psi^{-1}$ and $\Op$ is any quantization on $M$ (see \cite[Section 6.4]{Shubin-01} for instance). Here, $\Psi^m$ denotes the set of pseudodifferential operators of order $m \in \R$ and we will denote by $S^m$ the class of usual symbols of order $m$. Given $P \in \Psi^m$, we will denote by $\sigma_m$ its principal symbol. The following structure theorem is crucial in the sequel. It can be seen as a more intrinsic version of \cite[Theorem 2.1]{Sharafutdinov-Skokan-Uhlmann-05}.

\begin{theorem}
\label{theorem:microlocal-theorem} $\Pi_m$ is a pseudodifferential operator
of order $-1$ with principal symbol
\[ \sigma_m \coloneqq \sigma_{\Pi_m} : (x,\xi) \mapsto \frac{2\pi}{C_{n,m}} |\xi|^{-1}\pi_{\ker i_\xi} {\pi_m}_* \pi_m^* \pi_{\ker i_\xi},\]
with $C_{n,m} = \int_0^{\pi} \sin^{n-1+2m}(\varphi) d\varphi$.
\end{theorem}

\begin{proof}
The fact that $\Pi_m$ is pseudodifferential was proved in
\cite{Guillarmou-17-1}. All is left to compute is the principal symbol of
$\Pi_m$. According to the proof in \cite[Theorem 3.1]{Guillarmou-17-1}, we can only consider
the integral in time between $(-\eps,\eps)$. Namely, given $\chi \in C^\infty_c(\R)$ a smooth cutoff function around $0$ whose support is contained in $(-\eps,\eps)$, one has: 
\[
\begin{split}
\Pi_m = {\pi_m}_* & \int_{-\eps}^{\eps} \chi(t) e^{-tX} dt \pi_m^* \\
& - {\pi_m}_* R_0^+ \int_0^{+\infty} \chi'(t) e^{-tX} dt \pi_m^* - {\pi_m}_* R_0^- \int_{-\infty}^0 \chi'(t) e^{-tX} dt \pi_m^* \\ & + \left(1-\int_{-\infty}^{+\infty} \chi(t) dt\right) {\pi_m}_* \mathbf{1}\otimes\mathbf{1}~ \pi_m^*
\end{split}
\]
On the right-hand side, the last term is obviously smoothing. Following the same computations as in \cite[Theorem 3.1]{Guillarmou-17-1}, one can prove that the second and the third terms are also smoothing (this stems from an argument on the wavefront set of the kernel of these operators, using the fact that there are no conjugate points in the manifold). Thus, the pseudodifferential behaviour of the operator $\Pi_m$ is encapsulated by the first term whose kernel has a support living in a neighborhood of the diagonal in $M \times M$. In the following, $\eps > 0$ is chosen
small enough (less than the injectivity radius at the point $x$).

Let us consider a smooth section $f_1 \in C^\infty(M,\otimes^m_S
T^*M)$ defined in a neighborhood of $x \in M$ and $f_2 \in \otimes^m_S
T_x^*M$, then:
\[
\begin{split}
\langle \sigma_m(x_0,\xi) f_1, f_2 \rangle_{x_0} & = \lim_{h \rightarrow 0} h^{-1} e^{-iS(x_0)/h} \langle \Pi_m(e^{iS(x)/h}f_1),f_2 \rangle_{x_0} \\
& = \lim_{h \rightarrow 0} h^{-1} e^{-iS(x_0)/h} \langle \Pi \pi_m^*(e^{iS(x)/h}f_1),\pi_m^*f_2 \rangle_{L^2(S_{x_0}M)},
\end{split}
\]
where $\xi = dS(x) \neq 0$. Here, it is assumed that $\text{Hess}_{x} S$ is
non-degenerate. We obtain:
\[
\begin{split}
& \langle \sigma_m(x,\xi) f_1,f_2 \rangle_{x_0} \\
& = \lim_{h \rightarrow 0} h^{-1} \int_{\Ss^n} \int_{-\eps}^{+\eps} e^{i/h(S(\gamma(t))-S(x))} \pi_m^*f_1(\gamma(t),\dot{\gamma}(t)) \pi_m^*f_2(x_0,v) \chi(t) dt  dv \\
& = \lim_{h \rightarrow 0} h^{-1} \int_{\Ss^{n-1}} \left( \int_0^\pi  \int_{-\eps}^{+\eps} e^{i/h(S(\gamma(t))-S(x))} \pi_m^*f_1(\gamma(t),\dot{\gamma}(t)) \pi_m^*f_2(x_0,v) \sin^{n-1}(\varphi) \chi(t) dt d\varphi \right) du
\end{split}
\]
where $\chi$ is a cutoff function with support in $(-\eps,\eps)$, $\gamma$
is the geodesic such that $\gamma(0)=x,\dot{\gamma}(0)=v$ and we have
decomposed $v = \cos(\varphi)n + \sin(\varphi)u$ with $n=\xi^\sharp/|\xi|=
dS(x)^\sharp/|dS(x)|$, $u \in \Ss^{n-1}$. We apply the stationary phase
lemma \cite[Theorem 3.13]{Zworski-12} uniformly in the $u \in \Ss^{n-1}$
variable. For fixed $u$, the phase is $\Phi : (t,\varphi) \mapsto
S(\gamma(t))-S(x)$ so $\partial_t \Phi(t,\varphi) = dS(\dot{\gamma}(t))$.
More generally if $\tilde{\Phi} : (t,v) \mapsto S(\gamma(t))-S(x)$ denotes the map defined for any $v \in \Ss^{n}$, then
\[ \partial_v \tilde{\Phi}(t,v) \cdot V = d\pi(d\varphi_t(x,v)\cdot V), \qquad \forall V \in \V, \]
where $\V = \ker d\pi_0$, with $\pi_0 : SM \rightarrow M$ the natural
projection. Since $(M,g)$ has no conjugate points,
$d\pi(d\varphi_t(x,v))\cdot V \neq 0$ as long as $t \neq 0$ and $V \in \V
\setminus \left\{0\right\}$. And $dS(\dot{\gamma}(0)) = dS(\cos(\varphi)n +
\sin(\varphi)u) = \cos(\varphi)|dS(x)|=0$ if and only if $\varphi=\pi/2$.
So the only critical point of $\Phi$ is $(t=0,\varphi=\pi/2)$. Let us also
remark that
\[
\text{Hess}_{(0,\pi/2)} \Phi = \begin{pmatrix} \text{Hess}_{x} S(u,u) &
-|dS(x)| \\  -|dS(x)| & 0 \end{pmatrix} \] is non-degenerate with
determinant $-|\xi|^2$, so the stationary phase lemma can be applied and we
get:
\[
\begin{split}
 \int_0^\pi  \int_{-\eps}^{+\eps}& e^{i/h(S(\gamma(t))-S(x_0))} \pi_m^*f_1(\gamma(t),\dot{\gamma}(t)) \pi_m^*f_2(x_0,v) \sin^{n-1}(\varphi) dt d\varphi \\
 & \sim_{h \rightarrow 0}  2\pi h |\xi|^{-1} \pi_m^*f_1(x_0,u)\pi_m^*f_2(x_0,u).
\end{split}
\]
Eventually, we obtain:
\[ \langle \sigma_m(x,\xi) f_1,f_2 \rangle_{x_0} = \dfrac{2\pi}{|\xi|} \int_{\left\{\langle \xi,v\rangle=0\right\}} \pi_m^*f_1(v) \pi_m^*f_2(v) dS_\xi(v), \]
where $dS_\xi$ is the canonical measure induced on the $n-1$-dimensional
sphere $\Ss_xM \cap \left\{\langle \xi,v\rangle=0\right\}$. The sought
result then follows from Lemma \ref{lemma:symbole-projection}.
\end{proof}

\subsection{Ellipticity, injectivity on solenoidal tensors}

\begin{lemma}
\label{theorem:ellipticite-pim} The operator $\Pi_m$ is elliptic on solenoidal tensors,
that is there exists pseudodifferential operators $Q$ and $R$ of respective
order $1$ and $-\infty$ such that:
\[ Q\Pi_m = \pi_{\ker D^*} + R \]
\end{lemma}

\begin{proof}
We define
\[ \tilde{q}(x,\xi)= \left\{ \begin{array}{ll} 0, & \text{ on }\ran (\sigma j_\xi) \\ \dfrac{C_{n,m}}{2\pi} |\xi|(\pi_{\ker i_\xi} {\pi_m}_*\pi_m^* \pi_{\ker i_\xi})^{-1}, & \text{ on }\ker (i_\xi) \end{array} \right. \]
and $q(x,\xi) = (1-\chi(x,\xi))\tilde{q}(x,\xi)$ for some cutoff function $\chi \in C^\infty_c(T^*M)$ around the zero section. By construction, $\Op(q)\Pi_m = \pi_{\ker D^*} - R'$ with $R' \in
\Psi^{-1}$. Let $r' = \sigma_{R'}$ and define $a \sim \sum_{k=0}^\infty
r'^k$. Then $\Op(a)$ is a microlocal inverse for $\mathbbm{1}-R'$ that is $\Op(a)(\mathbbm{1}-R') \in \Psi^{-\infty}$. Since $R'D = 0$, we
obtain that $R'=R'\pi_{\ker D^*}$ and thus
\[ \underbrace{\Op(a)\Op(q)}_{=Q}\Pi_m = \Op(a)(\mathbbm{1}-R')\pi_{\ker D^*}= \pi_{\ker D^*} + R, \]
where $R$ is a smoothing operator.
\end{proof}

From now on, we assume that the X-ray transform is injective on solenoidal
tensors.

\begin{lemma}
\label{lemma:injectivite-pim} If $I_m$ is solenoidal injective, then $\Pi_m$ is
injective on $H_{\sol}^s(M,\otimes^m_S T^*M)$, for all $s \in \R$.
\end{lemma}

\begin{proof}
We fix $s \in \R$. We assume that $\Pi_m f = 0$ for some $f \in
H_{\sol}^s(M,\otimes^m_S T^*M)$. By ellipticity of the operator, we get that $f
\in C_{\sol}^\infty(M,\otimes^m_S T^*M)$. And:
\[ \begin{split} \langle \Pi_m f , f \rangle_{L^2} & = \langle \Pi \pi_m^* f, \pi_m^* f \rangle_{L^2} + \left(\int_{SM} \pi_m^*f d\mu\right)^2  \\
& = \langle (1 +\Delta_m)^{-s} \Pi \pi_m^* f, \pi_m^* f \rangle_{H^s} + \left(\int_{SM} \pi_m^*f d\mu\right)^2 = 0. \end{split} \]
Here, the Laplacian $\Delta_m$ is the one introduced in \S\ref{sssection:projection-solenoidal-tensors}. The scalar product on $H^s$ is $\langle f,h\rangle_{H^s} := \langle (1+\Delta_m)^{s/2} f, (1+\Delta_m)^{s/2}h\rangle_{L^2}$. By Lemma \ref{lemma:positivite-operateur-pi}, since $\langle \Pi
\pi_m^*f,\pi_m^*f\rangle \geq 0$, we obtain that $\int_{SM} \pi_m^*f d\mu =
0$. Moreover, $(1 +\Delta_m)^{-s} \Pi$ is bounded and positive (hence selfadjoint) on $H^s$ so
there exists a square root $R : H^s \rightarrow H^s$, that is a bounded
positive operator satisfying $(1 +\Delta_m)^{-s} \Pi = R^* R$, where $R^*$
is the adjoint on $H^s$. Then:
\[ \langle (1 +\Delta_m)^{-s} \Pi \pi_m^* f, \pi_m^* f \rangle_{H^s} = 0 = \|R \pi_m^* f\|^2_{H^s} \]
This yields $(1 +\Delta_m)^{-s} \Pi \pi_m^* f = 0$ so $\Pi \pi_m^* f = 0$.
By Theorem \ref{theorem:guillarmou}, there exists $u \in C^\infty(SM)$ such
that $\pi_m^* f = Xu$ so $f \in \ker I_m \cap \ker D^*$. By $s$-injectivity
of the X-ray transform, we get $f \equiv 0$.
\end{proof}

A direct consequence of Lemma \ref{lemma:injectivite-pim} and Theorem
\ref{theorem:ellipticite-pim} is the

\begin{theorem}
If $I_m$ is solenoidal injective, then there exists a pseudodifferential operator $Q'$ of order $1$ such that:
$Q'\Pi_m = \pi_{\ker D^*}$.
\end{theorem}

\begin{proof}
The operator $\Pi_m$ is elliptic of order $-1$ on $\ker D^*$, thus
Fredholm as an operator $H^s_{\sol}(M,\otimes^m_S T^*M) \rightarrow H^{s+1}_{\sol}(M,\otimes^m_S T^*M)$ for all $s \in \R$. It is selfadjoint on $H_{\sol}^{-1/2}(M,\otimes^m_S T^*M)$, thus Fredholm of index $0$ (the index being independent of the Sobolev space considered, see \cite[Theorem 8.1]{Shubin-01}), and injective, thus
invertible on $H^s_{\sol}(M,\otimes^m_S T^*M)$. We multiply the equality $Q\Pi_m = \pi_{\ker D^*}
+ R$ on the right by $Q' \coloneqq \pi_{\ker D^*} \Pi_m^{-1} \pi_{\ker
D^*}$:
\[ Q\Pi_mQ' = Q \underbrace{\Pi_m \pi_{\ker D^*}}_{=\Pi_m} \Pi_m^{-1} \pi_{\ker D^*} = Q \pi_{\ker D^*} = Q' + RQ' \]
As a consequence, $Q' = Q \pi_{\ker D^*} + \text{smoothing}$ so it is a
pseudodifferential operator of order $1$. And $Q'\Pi_m = \pi_{\ker D^*}$.
\end{proof}

This yields the following stability estimate:

\begin{lemma}
\label{lemma:stabilite-pim} If $I_m$ is solenoidal injective, then for all $s \in \R$, there exists a constant
$C\coloneqq C(s) > 0$ such that:
\[ \forall f \in H^s_{\sol}(M,\otimes^m_S T^*M), \qquad \|f\|_{H^s} \leq C\|\Pi_mf\|_{H^{s+1}} \]
\end{lemma}

%

\subsection{Stability estimates for the X-ray transform}

Before going on with the proof of Theorem \ref{theorem:xray}, let us recall the definition \emph{Hölder-Zygmund spaces}. Let $\psi \in C^\infty_c(\R)$ be a smooth cutoff function with support in $[-2,2]$ and such that $\psi \equiv 1$ on $[-1,1]$. For $j \in \N$, we introduce the functions $\varphi_j \in C^\infty_c(T^*M)$ defined by $\varphi_0(x,\xi) := \psi(|\xi|), \varphi_j(x,\xi) := \psi(2^{-j}|\xi|) - \psi(2^{-j+1}|\xi|)$, for $j \geq 1$ with $(x,\xi) \in T^*M$, $|\cdot|$ being the norm induced by $g$ on the cotangent bundle. Since $\varphi_j$ is a symbol in $S^{-\infty}$, one observes that the operators $\Op(\varphi_j)$ are smoothing.

For $s \in \R$, we define $C^s_*(M)$, the \emph{Hölder-Zygmund space of order $s$} as the completion of $C^\infty(M)$ with respect to the norm
\[
 \|u\|_{C^s_*} := \sup_{j \in \N} 2^{js} \|\Op(\varphi_j)u\|_{L^\infty},
\]
and we recall (see \cite[Appendix A, A.1.8]{Taylor-91} for instance) that a pseudodifferential operator $P \in \Psi^m(M)$ of order $m \in \R$ is bounded as an operator $C^{s+m}_*(M) \rightarrow C_*^s(M)$, for all $s \in \R$. Note that the previous definition of Hölder-Zygmund spaces can be easily generalized to sections of a vector bundle. When $s \in (0,1)$, it is a well-known fact that the space $C^s_*(M)$ coincide with $C^s(M)$, the space of Hölder-continuous functions, with equivalent norms $\|u\|_{C^s_*} \asymp \|u\|_{C^s}$. The Hölder-Zygmund spaces correspond to the Besov spaces $B^s_{q,r}(M)$ with $q = r = +\infty$ while the Sobolev spaces $H^s(M)$ correspond to the choice $q=r=2$. Here:
\[
\|u\|_{B^s_{q,r}} := \left(\sum_{j=0}^{+\infty} \|2^{sj}\Op(\varphi_j)u\|_{L^q}^r \right)^{1/r}
\]
In particular, Lemma \ref{lemma:stabilite-pim} can be upgraded to:
\begin{lemma}
\label{lemma:stabilite-pim2}  If $I_m$ is solenoidal injective, then for all $s \in \R$, there exists a constant
$C\coloneqq C(s) > 0$ such that:
\[ \forall f \in C^s_{*,\sol}(M,\otimes^m_S T^*M), \qquad \|f\|_{C_*^s} \leq C\|\Pi_mf\|_{C_*^{s+1}} \]
\end{lemma}

Eventually, we will need this last result:

\begin{lemma}
\label{lemma:boundedness-pi}
For all $s > 0$, the operator $\Pi : C_*^s(SM) \rightarrow C_*^{-s-(n+1)/2}(SM)$ is bounded.
\end{lemma}

\begin{proof}
Fix $\eps > 0$ small enough. Then:
\[
C_*^s \hookrightarrow H^{s-\eps} \overset{\Pi}{\rightarrow} H^{-s+\eps} \hookrightarrow C_*^{-s-(n+1)/2+\eps} \hookrightarrow C_*^{s-(n+1)/2},
\]
by Sobolev embeddings.
\end{proof}

We can now deduce from the previous work the stability estimate of Theorem \ref{theorem:xray}.

\begin{proof}[Proof of Theorem \ref{theorem:xray}]
We assume that $f \in C_{\sol}^\alpha(M, \otimes^m_S T^*M)$ is such that $\|f\|_{C^\alpha} \leq 1$. By Theorem \ref{theorem:livsic-approche}, we can write $\pi_m^*f = Xu+h$, with $u, Xu, h \in C^{\alpha'}$, where $0 < \alpha' < \alpha$ and $\|h\|_{C^{\alpha'}} \lesssim \|I_mf\|^\tau_{\ell^\infty}$.

We have:
\[
\begin{array}{lll}
\|f\|_{C_*^{-1-\alpha'-(n+1)/2}} & \lesssim \|\Pi_2f\|_{C_*^{-\alpha'-(n+1)/2}} & \text{ by Lemma \ref{lemma:stabilite-pim2}} \\
& \lesssim \|\Pi \pi_2^* f\|_{C_*^{-\alpha'-(n+1)/2}} & \\
& \lesssim \|\Pi (Xu+h)\|_{C_*^{-\alpha'-(n+1)/2}} & \\
& \lesssim \|\Pi h\|_{C_*^{-\alpha'-(n+1)/2}}& \\ 
& \lesssim \|h\|_{C^{\alpha'}} & \text{ by Lemma \ref{lemma:boundedness-pi}} \\
& \lesssim \|I_2f\|_{\ell^\infty}^{\tau} & \text{ by Theorem \ref{theorem:livsic-approche}}
\end{array}
\]
Using $\|f\|_{C^\alpha} \leq 1$ and interpolating $C^\beta$ between $C^{-1-\alpha'-(n+1)/2}$ and $C^{\alpha}$, one obtains the sought result.

\end{proof}

\bibliographystyle{alpha}
\bibliography{Biblio}

\end{document}